\documentclass[a4paper]{amsart}

\usepackage{amsthm,amsfonts,amsmath, amssymb}
\usepackage[capitalise]{cleveref}
\usepackage{tikz}
\usepackage{cmbright}
\usepackage{geometry}                % See geometry.pdf to learn the layout options. There are lots.
\usepackage{anysize}
\marginsize{2cm}{2cm}{2cm}{2cm}%\marginsize{left}{right}{top}{bottom}

\renewcommand{\baselinestretch}{1.25}
\usepackage{mathtools}

\DeclareMathOperator{\ver}{vert}

\def\h{\frac{1}{2}}

\def\TA{{T\hskip-3.11pt A}}
\def\r{\mathbb{R}}
\usepackage[sort&compress,square,comma,numbers]{natbib}

\usepackage{graphicx,float,enumerate}
\usepackage[ruled,boxed,commentsnumbered,norelsize]{algorithm2e}
\usepackage{verbatim}
\usepackage[hyphens]{url}
\usepackage{epstopdf}
\usepackage[notref,notcite,final]{showkeys}

\renewcommand*\showkeyslabelformat[1]{%
\fbox{\parbox[t]{\marginparwidth}{\raggedright\normalfont\tiny\url{#1}}}
}
\usepackage[capitalise]{cleveref}
\usepackage{subfigure}

\newtheorem{theorem}{Theorem}[section]
\newtheorem{corollary}[theorem]{Corollary}
\newtheorem{lemma}[theorem]{Lemma}
\newtheorem{proposition}[theorem]{Proposition}

\theoremstyle{definition}

\theoremstyle{remark}
\newtheorem{rmk}[theorem]{Remark}

\crefname{rmk}{Remark}{Remarks}
\date{\today}
\title{{The excess degree of a polytope}}

\author{Guillermo Pineda-Villavicencio}
\address{Centre for Informatics and Applied Optimisation, Federation University Australia}
\email{\texttt{work@guillermo.com.au}}

\author{Julien Ugon}
\address{School of Information Technology, Deakin University}
\email{\texttt{julien.ugon@deakin.edu.au}}
\thanks{Research by Ugon was supported by ARC discovery project DP180100602.}

\author{David Yost}
\address{Centre for Informatics and Applied Optimisation, Federation University Australia}
\email{\texttt{d.yost@federation.edu.au}}

\keywords{polytope; Minkowski decomposability; f-vector; polytope graph; excess degree}
\subjclass[2010]{Primary 52B05; Secondary 52B11}

\begin{document}

\maketitle

\begin{abstract}
We define the excess degree $\xi(P)$ of a $d$-polytope $P$ as $2f_1-df_0$, where $f_0$ and $f_1$ denote the number of vertices and edges, respectively. This parameter measures how much $P$ deviates from being simple.

It turns out that the excess degree of a $d$-polytope does not take every natural number: the smallest possible values are $0$ and $d-2$, and the value $d-1$ only occurs when $d=3$ or 5. On the other hand, for fixed $d$, the number of values not taken by the excess degree is finite if $d$ is odd, and the number of even values not taken by the excess degree is finite if $d$ is even.

The excess degree  is then applied in three different settings. It is used to show that polytopes with small excess (i.e. $\xi(P)<d$) have a very particular structure: provided $d\ne5$, either there is a unique nonsimple vertex, or every nonsimple vertex has  degree $d+1$. This implies that such polytopes behave in a similar manner to simple polytopes in terms of Minkowski decomposability: they are either decomposable or pyramidal, and their duals are always indecomposable.
Secondly, we characterise completely the decomposable $d$-polytopes with $2d+1$ vertices (up to combinatorial equivalence). And thirdly all pairs $(f_0,f_1)$, for which there exists a 5-polytope with $f_0$ vertices and $f_1$ edges, are determined.

\end{abstract}

\section{Introduction}
This paper revolves around the excess degree of a $d$-dimensional polytope $P$, or simply $d$-polytope, and some of its applications. We define the {\it excess degree} $\xi(P)$, or simply excess,  of a $d$-polytope $P$ as the sum of the excess degrees of its vertices, i.e.
\[\xi(P)=2f_1-df_0=\sum_{u\in \ver P}(\deg u -d).\]
Here as usual $\deg u$ denotes the {\it degree} of a vertex $u$, i.e~the number of edges of $P$ incident with the vertex;  $\ver P$ denotes the set of vertices of $P$; and $f_0$ and $f_1$ denote the number of vertices and edges of the polytope. This concept is implicit in some results in \cite[\S6]{Smi87}, but has not been studied consistently before.

Our first substantial result, in \S3, is the excess theorem: the smallest values of the excess degree of $d$-polytopes are 0 and $d-2$; clearly a polytope is simple if, and only if, its excess degree is 0. Note that for fixed $d$ and $f_0$, the possible values of the excess are either all even or all odd. We further show that if $d$ is even, the excess degree takes every even natural number from $d\sqrt{d}$ onwards; while, if $d$ is odd, the excess degree takes every natural number from $d\sqrt{2d}$ onwards. So, for a fixed $d$, only a finite number of gaps in the values of the excess of a $d$-polytope  are possible.

In \S4, we study $d$-polytopes with excess strictly less than $d$, and establish some similarities with simple polytopes. In particular, they are either pyramids or decomposable.
A polytope $P$ is called {\it (Minkowski) decomposable} if it can be written as the Minkowski sum of two polytopes, neither of which are homothetic to $P$; otherwise it is {\it indecomposable}. See \cite[Chap. 15]{Gru03} for a more detailed account and historical references. As usual, the {\it(Minkowski) sum} of two polytopes $Q+R$ is defined to be $\{x+y: x\in Q, y\in R\}$, and a polytope $P$ is said to be {\it homothetic} to a polytope $Q$ if  $Q=\lambda P+t$ for $\lambda> 0$ and $t\in \mathbb{R}^d$.
Polytopes with  excess $d-2$ exist in all dimensions, but their structure is quite restricted: either there is a unique nonsimple vertex, or there is a $(d-3)$-face containing only vertices with excess degree one. Polytopes with excess $d-1$  are in one sense even more restricted: they can exist only if $d=3$ or 5. However a 5-polytope with excess 4 may also contain vertices with excess degree 2. On the other hand, $d$-polytopes with excess degree $d$ or $d+2$ are exceedingly numerous.

In \S5, we characterise all the decomposable $d$-polytopes with $2d+1$ or fewer vertices; this incidentally proves that a conditionally decomposable $d$-polytope must have at least $2d+2$ vertices.

The final application, in \S7, is the completion of the $(f_0,f_1)$ table for $d\le 5$; that is, we give all the possible values of $(f_0,f_1)$ for which there exists a $d$-polytope with $d= 5$, $f_0$ vertices and $f_1$ edges. The solution of this problem for $d\le4$ was already well known \cite[Chap. 10]{Gru03}. The same result has recently  been independently obtained by Kusunoki and Murai \cite{KusMur17}. Our proof requires some  results of independent interest; in particular, a characterisation of the 4-polytopes with 10 vertices and minimum number of edges (namely, 21); this is completed in \S6. We have more comprehensive results characterising polytopes with a given number of vertices and  minimum possible number of edges, details of which will appear elsewhere \cite{PinUgoYos16}.

{Most of our results tacitly  assume that the dimension $d$ is at least 3. When $d=2$, all polytopes are both simple and simplicial, and the reader can easily see which theorems remain valid in this case.}

\section{Background: some special polytopes and previous results}
	
\subsection{Some basic results on polytopes} In this subsection we group a number of basic results on polytopes which we will use throughout the paper. Recall that a {\it facet} of a $d$-polytope is a face of dimension $d-1$; a {\it ridge} is a face of dimension $d-2$; and a {\it subridge} is a face of dimension $d-3$. A fundamental property of polytopes is that every ridge is contained in precisely two facets.

Recall that a vertex $u$ in a $d$-polytope $P$ is called {\it simple} if its degree  in $P$ is precisely $d$; equivalently if it is contained in exactly $d$ facets. Otherwise it is {\it nonsimple}. Note that a nonsimple vertex in a polytope $P$ may be simple in a proper face of  $P$; we often need to make this distinction. A polytope is {\it simple} if every vertex is simple; otherwise it is {\it nonsimple}.

Let $H$ be a hyperplane intersecting the interior of $P$ and containing no vertex of $P$, and let $H^+$ be one of the two closed half-spaces bounded by $H$. Set $P':=H^+\cap P$. If the vertices of $P$ not in $H^+$ are the vertices of a face $F$, then the polytope $P'$ is said to be obtained by {\it truncation} of the face $F$ by $H$. We often say that $P$ has been {\it sliced} or {\it cut} at $F$.
We will then call $H\cap P$ the {\bf underfacet} corresponding to $F$. Klee \cite{Kle78} called this the face figure, but that term is sometimes used for a different concept \cite[p. 71]{Zie95}. When $F$ has dimension 0, the underfacet is simply the vertex figure.

\begin{lemma}\label{lem:face-all-neighbours-simple} Let $P$ be a $d$-polytope and let $F$ be a face of $P$. Suppose the neighbours outside $F$ of every vertex of $F$ are all simple in $P$. Let $H$ be a hyperplane as above, and let $P':=H^+\cap P$ be obtained by truncation of $F$ by $H$. Then every vertex in the facet $H\cap P$ is simple in $P'$, and thus the facet $H\cap P$ is a simple $(d-1)$-polytope.\end{lemma}
\begin{proof}
Every vertex $u_e$ in the facet $H\cap P$ is the intersection of $H$ and an edge $e$ of $P$ outside $F$ but incident to a vertex of $F$.  Consequently, the facets of $P'$ containing the vertex $u_e$ in $H\cap P$ are precisely the facets of $P$ containing the edge $e$ plus the facet $H\cap P$. Since the edge $e$ is contained in exactly $d-1$ facets of $P$, the vertex $u_e$ in $P'$ is contained in exactly $d$ facets of $P'$, thus is simple in $P'$.
\end{proof}

An interesting corollary of  \cref{lem:face-all-neighbours-simple} arises when $F$ is a vertex and reads as follows.
\begin{lemma}\label{lem:vertex-all-neighbours-simple} Let $P$ be a $d$-polytope and let $v$ be a vertex whose neighbours are all simple in $P$. Then the vertex figure of $P$ at $v$ is a simple $(d-1)$-polytope.\end{lemma}

The next few results are simple but useful.

\begin{rmk}\label{rmk:two-vertices}
For any two faces $F,G$ of a polytope, with $F$ not contained in $G$, there is a facet containing $G$ but not $F$. In particular, for any two distinct vertices of a polytope, there is a facet containing one but not the other.
\end{rmk}

\begin{lemma}\label{lem:simpleNonsimpleVertex} Let $P$ be a $d$-polytope and let $v$ be a vertex simple in a facet $F$. Suppose that every facet containing $v$ intersects $F$ at a ridge. Then $v$ is a simple vertex in $P$.
\end{lemma}
\begin{proof}  The other facets containing $v$ are given by the ridges of $F$ containing $v$, whose number is $d-1$ since $v$ is simple in $F$. Thus, in total there are $d$ facets containing $v$.
\end{proof}

\begin{lemma}\label{lem:simpleVertexOutside}
Let $P$ be a polytope, $F$ a facet of $P$ and $u$ a nonsimple vertex of $P$ which is contained in $F$. If $u$ is adjacent to a simple vertex $x$ of $P$ in $P\setminus F$, then $u$ must be adjacent to another vertex in $P\setminus F$, other than $x$.
\end{lemma}

\begin{proof}
Since $u$ is nonsimple, it is contained in at least $d+1$ facets of $P$. The edge $ux$ is contained in exactly $d-1$ facets of $P$, since $x$ is simple. Hence there are at least two edges of $u$ outside $F$, as desired.
\end{proof}

\begin{lemma}\label{lem:Non-DisjointFacets}
Let $F$ and $G$ be any two distinct nondisjoint facets of a $d$-polytope $P$  and let $j:=\dim F\cap G$.
\begin{enumerate}[(i)]
\item Every vertex in $F\cap G$ has excess degree at least $(d-2-j)$.
\item The total excess degree of $P$ is at least $\max\{\xi(F),\xi(G),\xi(F\cap G)\}+(d-2-j)(j+1)$.
\item If $F\cap G$ is not a ridge, then $P$ is not simple.
\end{enumerate}
\end{lemma}

\begin{proof}
Set $R:=F\cap G$. Then, for any vertex $u$ in $R$, the degrees of $u$ in $R$, $F\setminus R$ and $G\setminus R$ are at least $j$, $d-1-j$ and $d-1-j$, respectively. So the total degree of $u$ in $F_1\cup F_2$ is at least $2d-2-j$. This implies that the excess degree of each vertex in $R$ is at least $d-2-j$. Since there are at least $j+1$ vertices in $R$, the total excess degree of $P$ is at least $(d-2-j)(j+1)$.
\end{proof}

 In particular, if two facets of a $d$-polytope intersect in a face $K$ with $\dim K<d-2$, then every vertex in $K$ is nonsimple in the polytope.
We will call a polytope {\bf semisimple}  if every pair of facets is either disjoint, or intersects in a ridge. Clearly every simple polytope is semisimple. The converse is false. A pyramid over $\Delta_{2,2}$, which is defined in the next subsection, is one  counterexample. In fact, it is the easiest example, as the next result shows. Some related examples are discussed in \S4.1.

\begin{lemma}\label{lem:semilowdim}
If $d\le4$, every semisimple $d$-polytope is simple.
\end{lemma}
\begin{proof}
Suppose that $P$ is semisimple but not simple. Then its dual $P^*$ is not simplicial, but every facet of $P^*$ must be 2-neighbourly. Thus the facets of $P^*$ have dimension at least 4, and $P$ must have dimension at least 5.
\end{proof}

\begin{lemma}\label{lem:facets-intersecting-at-ridges} Let $P$ be a semisimple $d$-polytope. Then every facet containing a nonsimple vertex of $P$ is also nonsimple. Furthermore, each nonsimple vertex of $P$ is nonsimple in each of the facets containing it.\end{lemma}

\begin{proof} If $d=3$, then \cref{lem:simpleNonsimpleVertex} gives at once that the polytope is simple. So we let $d>3$. Let $u$ be a nonsimple vertex of $P$ and let $F$ be some facet containing it. Since every other facet containing $u$ intersects $F$ at a ridge, there are at least $d$ ridges of $F$ containing $u$, which implies that $u$ is nonsimple in $F$, and $F$ is therefore nonsimple.
\end{proof}

\subsection{{Taxonomy} of polytopes}

In this subsection we introduce or recall
families of polytopes which are important for this work.

There is a 3-polytope with six vertices, ten edges and six facets (four quadrilaterals and two triangles), which can fairly be described as the simplest polyhedron with no widely accepted name. Kirkman \cite[p. 345]{Kir63} called it a 2-ple zoneless monaxine heteroid, having an amphigrammic axis, but this nomenclature never caught on. More recently, Michon \cite{Mic07} has descriptively called it a tetragonal antiwedge. We will use his terminology here, and abbreviate it to $\TA$; see \cref{fig:polytopes}(d). This humble example will naturally appear several times in this paper.

\begin{figure}
\includegraphics{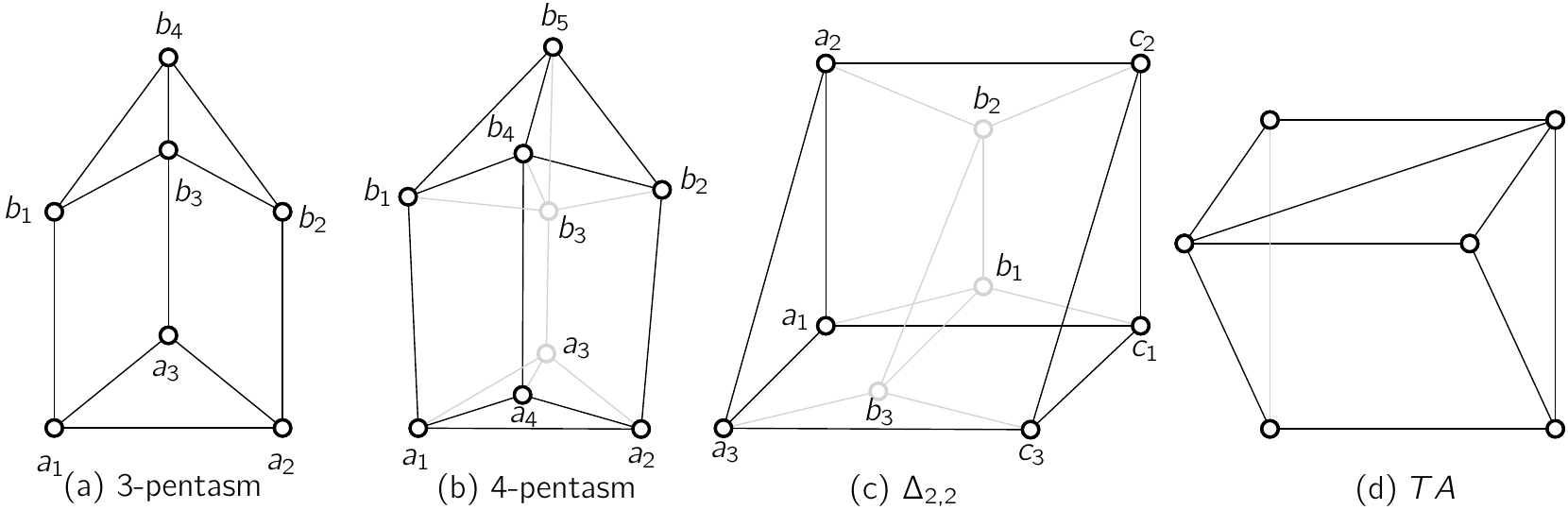}
\caption{Solid realisations or Schlegel diagrams of polytopes.}\label{fig:polytopes}
\end{figure}

Call {\it simplicial $d$-prism}  any prism whose base is a $(d-1)$-simplex. We will often refer to these simply as prisms. They each have $2d$ vertices, $d^2$ edges, and $d+2$ facets. Being simple, they have excess degree 0.

Define a {\it triplex} as any multifold pyramid over a simplicial prism. More precisely, we will call a $(d-k)$-fold pyramid over the simplicial $k$-prism a {\it $(k,d-k)$-triplex}, and denote it by  $M_{k,d-k}$; here $1\le k\le d$. This triplex clearly has $d+k$ vertices. Of course $M_{1,d-1}$ is just a simplex, and $M_{d,0}$ is a simplicial $d$-prism. It is well known (see \cite{McM71} and  \cite[Ch. 10]{Gru03}) and easily checked, that triplices (other than the simplex) have $d+2$ facets. They were studied in some detail in \cite{PinUgoYos15}. It is routine to check that $M_{k,d-k}$ has excess degree $(k-1)(d-k)$.

More generally, denote by $\Delta_{m,n}$ the cartesian product of a simplex in $\r^m$ and a simplex in $\r^n$. Alternatively, $\Delta_{m,n}$ can be described as the sum of an $m$-dimensional simplex and an $n$-dimensional simplex lying in complementary subspaces of $\r^{m+n}$. The polytope $\Delta_{m,n}$ is a simple $(m + n)$-dimensional polytope with $(m + 1)(n + 1)$ vertices, $\frac{1}{2}(m+n)(m+1)(n+1)$ edges and $m+n+2$ facets. Likewise $\Delta_{m,n,p}$ will denote  the cartesian product of three simplices, $m$-dimensional, $n$-dimensional and  $p$-dimensional, respectively.

\begin{rmk}\label{rmk:dplus2Facets} Any polytope with $d+2$ facets is combinatorially isomorphic to an $r$-fold pyramid over $\Delta_{m,n}$ for some values of $m,n,r$; see ~\cite[Sec~6.1]{Gru03} and \cite[p.~352]{McM71}.
In particular, a simple polytope with $d+2$ facets must be $\Delta_{m,n}$ for some $m,n$ with $m+n=d$.
\end{rmk}

The $d$-dimensional {\it pentasm}, or just $d$-pentasm, was defined in \cite[\S4]{PinUgoYos15} as the Minkowski sum of a simplex and a line segment which is parallel to one triangular face,
but not parallel to any edge, of the simplex; or any polytope combinatorially equivalent to it. The same polytope is obtained by truncating a simple vertex of the triplex $M_{2,d-2}$. Pentasms have $2d+1$ vertices, $d^2+d-1$ edges, and hence excess degree $d-2$; see \cref{fig:polytopes}(a)-(b) for drawings of them.

\begin{rmk}[Facets of a $d$-pentasm] \label{rmk:Pentasm-Facets}The facets of $d$-pentasm, $d+3$ in total, are as follows.
\begin{itemize}
\item $d-2$ copies of a $(d-1)$-pentasm,
\item 2 simplicial prisms, and
 \item 3 simplices.
\end{itemize}
\end{rmk}

We can consider the pentagon as a 2-dimensional pentasm.

Take any polytope $P$ with dimension $d\ge4$ and stack  a vertex $v_0$ beyond one facet and move it slightly so that $v$ is contained in the affine hull of $\ell$ other facets
($0\le \ell\le d-3$). By \emph{stacking} we mean as usual adding a vertex beyond a facet of $P$ and taking the convex hull. A point $w$ is {\it beyond} a face $R$ of $P$ if the facets of $P$ containing $R$ are precisely those that are visible from $w$, where  a facet $F$ of $P$ is said to be {\it visible} from the point $w$ if $w$ belongs to the open half-space determined by the affine hull of $F$ which does not meet $P$. All the $d-2$ polytopes constructed in this manner have the same graph, some have even the same $(d-3)$-skeleton.

If $P$ is taken to be a simplicial $d$-prism, and we stack a vertex on one of the simplex facets, we will call the resulting polytopes {\it capped $d$-prisms}.
Call the {\it extra} vertex $v_0$, and denote by $k$ the minimum dimension of any face of the simplicial $d$-prism whose affine hull contains it. The combinatorial type of such a polytope depends on the value of $k$; let us denote it by $CP_{k,d}$. Note that $k\ge1$, otherwise $v_0$ would be a vertex of the prism. For $k=1$, the capped prism $CP_{1,d}$ will be (combinatorially) just another prism. For $k=2$, $CP_{2,d}$ is a pentasm, with $d^2+d-1$ edges. For $3\le k\le d$ and fixed $d\ge4$, the  $d-2$ polytopes $CP_{k,d}$ are  combinatorially distinct, although their graphs are all isomorphic, with $2d+1$ vertices and $d^2+d$ edges. In particular, they all have excess degree $d$. However $CP_{k,d}$ has $d+k+1$ facets, so their $f$-vectors  are all distinct.

\begin{rmk}[Facets of the Capped Prism $CP_{k,d}$] \label{rmk:CappedPrismFacets} For $3\le k\le d$, the facets of the Capped Prism $CP_{k,d}$, $d+k+1$ in total, are as follows.
\begin{itemize}
\item $d-k$ copies of  $CP_{k,d-1}$,
\item $k$ simplicial prisms, and
\item $k+1$ simplices.
\end{itemize}
\end{rmk}

Denote by $A_d$ a polytope obtained by slicing a one nonsimple vertex of  a $(2,d-2)$-triplex. This polytope can be also realised as a prism over a $(d-3)$-fold pyramid over a quadrilateral. These polytopes have $2d+2$ vertices and excess degree $2d-6$.

\begin{rmk}[Facets of $A_d$] \label{rmk:Ad-Facets}The facets of the $d$-polytope $A_d$, $d+3$ in total, are as follows.
\begin{itemize}
\item $d-3$ copies of $A_{d-1}$,
\item 4 simplicial prisms, and
 \item 2 copies of $M_{2,d-3}$.
\end{itemize}
\end{rmk}

Denote by $B_d$  a polytope obtained by truncating a simple vertex of a $(3,d-3)$-triplex. The polytope $B_3$ is the well known {\it 5-wedge}; it can be obtained as the wedge \cite[pp. 57--58]{KleWal67} at an edge over a pentagon. The polytope $B_d$ can also be realised as follows: the convex hull of $B_3$ and a simplicial $(d-3)$- prism $R$, where each vertex of one copy of the $(d-4)$-simplex in $R$ is connected to each of the three vertices in a triangle of $B_3$, and  each vertex of the other copy of the $(d-4)$-simplex in $R$ is connected to each of the remaining five vertices of $B_3$.     These polytopes also have $2d+2$ vertices and excess degree $2d-6$.

\begin{rmk}[Facets of $B_d$]\label{rmk:Bd-Facets} The facets of the $d$-polytope $B_d$, $d+3$ in total, are as follows.
\begin{itemize}
\item $d-3$ copies of $B_{d-1}$,
\item 2 simplices,
\item 1 simplicial prism,
 \item 1 copy of $M_{2,d-3}$, and
 \item 2 pentasms.
\end{itemize}
\end{rmk}

Denote by $C_d$ a polytope obtained by slicing one {\it simple edge}, i.e. an edge whose endvertices are both simple, of a  $(2,d-2)$-triplex. It has $3d-2$ vertices and excess degree $d-2$.

\begin{rmk}[Facets of $C_d$]\label{rmk:Cd-Facets} The facets of the $d$-polytope $C_d$, $d+3$ in total, are as follows.
\begin{itemize}
\item $d-2$ copies of $C_{d-1}$,
\item 3 simplicial prisms,
 \item 1 copy of $\Delta_{2,d-3}$, and
 \item 1 simplex.
\end{itemize}

For consistency here, we can define $C_2$ as a quadrilateral.
\end{rmk}

Denote by $\Sigma_d$ a certain polytope which can be expressed as the Minkowski sum of two
simplices. One concrete realisation of it is given by the convex hull of
 \[\{0,e_1,e_1+e_k,e_2,e_2+e_k,e_1+e_2,e_1+e_2+2e_k: 3\le k\le d\},\]
where $\{e_i\}$ is the standard basis of $\mathbb{R}^d$. It is easily shown to have $3d-2$ vertices; of these, one has excess degree $d-2$, and the rest are simple.
\begin{rmk}[Facets of $\Sigma_d$]\label{rmk:Sigmad-Facets} The facets of the $d$-polytope $\Sigma_d$, $d+3$ in total, are as follows.
\begin{itemize}
\item $d-1$ copies of $\Sigma_{d-1}$,
\item 2 simplicial prisms,
 \item 2 simplices.
\end{itemize}

For consistency, we can also define $\Sigma_2$ as a quadrilateral.
\end{rmk}

Diagrams of $A_4, B_4, C_4$ and $\Sigma_4$ appear in \cref{fig:10v-21e}(b) in \S6.

Let us denote by $\Gamma_{m,n}$ the result of truncating one vertex from $\Delta_{m,n}$. This is a simple $(m + n)$-dimensional polytope with $mn + 2m+2n$ vertices and $m+n+3$ facets.

\begin{rmk}[Facets of $\Gamma_{m,n}$]\label{rmk:Gamma-Facets} The facets of the $d=m+n$-polytope $\Gamma_{m,n}$, $d+3$ in total, are as follows.
\begin{itemize}
\item $m$ copies of $\Gamma_{m-1,n}$,
\item $n$ copies of $\Gamma_{m,n-1}$,
\item $1$ copy of $\Delta_{m-1,n}$,
\item $1$ copy of $\Delta_{m,n-1}$,
 \item 1 simplex.
\end{itemize}
\end{rmk}

Denote by $J_d$ the special case $\Gamma_{d-1,1}$, i.e. the polytope obtained by slicing one vertex of a  simplicial $d$-prism; it clearly has $3d-1$ vertices. Observe that $J_2$ is a pentagon, and that $B_3$ coincides with $J_3$; so $J_d$ can be considered a generalisation of the 5-wedge.

\begin{rmk}[Facets of $J_d$]\label{rmk:Jd-Facets} The facets of the $d$-polytope $J_d$, $d+3$ in total, are as follows.
\begin{itemize}
\item $d-1$ copies of $J_{d-1}$,
\item 2 simplicial prisms,
 \item 2 simplices.
\end{itemize}
\end{rmk}

It can be shown that, for  $d\ne3, 7$, $J_d$ is the unique simple polytope with $3d-1$ vertices; see \cref{lem:CountSimpleP} below.

Some of the examples we have just defined coincide in low dimensions. In particular, $A_3=\Delta_{1,1,1}$, $B_3=J_3$ and $C_3=\Sigma_3$. By definition $\Delta_{1,d-1}=M_{d,0}$ and $J_d=\Gamma_{d-1,1}$.

With the fundamental examples now elucidated, we can characterise all the simple $d$-polytopes with up to $3d$ vertices, which will  be  useful later in the paper. First we need  the following result, of independent interest.
We omit the case $d=2$ in the statement of the next two results, since all polygons are simple.

\begin{proposition}\label{prop:simple-prisms-simplices}
A simple $d$-polytope $P$ in which every facet is a prism or a simplex is either  a $d$-simplex, a simplicial $d$-prism, $\Delta_{1,1,1}$ or $\Delta_{2,2}$.
\end{proposition}

\begin{proof}
If $P$ is not a simplex, then it is not simplicial, so assume that some facet $F$ is a prism. Choose a ridge $R$ in $F$ which is also a prism, and let $G$ be the other facet containing $R$. Then $G$ is also a prism and every vertex in $R$ has degree $d$ in the subgraph $F\cup G$. So no vertex in $R$ is adjacent to any vertex outside $F\cup G$. If there are no vertices outside $F\cup G$, then $P$ must be a simplicial prism. Otherwise, the removal of the four vertices in $(F\cup G)\setminus R$ from the graph of $P$ will  disconnect it. By Balinski's Theorem  \cite[Thm.~3.14]{Zie95}, we then have $d\le4$. Denote by $t$ the number of triangular 2-faces of $P$, and by $q$ the number of quadrilateral 2-faces.

If $d=3$,  simplicity implies $2f_1=3f_0$. Then $f_0$ is even and
$$0\le t=4(t+q)-(3t+4q)=4f_2-2f_1=4(f_1-f_0+2)-2f_1=2f_1-4f_0+8=3f_0-4f_0+8=8-f_0.$$
Thus either $f_0=4$ and $P$ is a simplex, or $f_0=6$ and $P$ is a prism, or $f_0=8$ and $P$ is a cube.

If $d=4$, simplicity implies $2f_1=4f_0$, whence $f_2=f_1+f_3-f_0=f_0+f_3$. Simplicity also ensures that every edge belongs to precisely three 2-faces. By hypothesis, every facet contains at least 2 triangular faces, and every triangular face, being a ridge, belongs to precisely 2 facets. Thus $t\ge f_3$. The Lower Bound Theorem \cite{Bar73} tells us $f_0\ge3f_3-10$. Then
\begin{eqnarray*}
\h f_3\le\h t& = & 2(t+q)-\h(3t+4q)\\
& = & 2f_2-\frac{3}{2}f_1 \\
& = & 2(f_0+f_3)-3f_0 \\
& = & 2f_3-f_0 \\
& \le & 2f_3-(3f_3-10) \\
& = &10-f_3.
\end{eqnarray*}
Thus $\frac{3}{2}f_3\le10$. So either $f_3=5$ and $P$ is a simplex, or $f_3=6=d+2$, and \cref{rmk:dplus2Facets} tells us that $P$ has the form $\Delta_{m,n}$ for some $m,n$ with $m+n=4$.
\end{proof}

\begin{lemma}\label{lem:CountSimpleP} Up to combinatorial equivalence,
\begin{enumerate}[(i)]
\item the simplex $\Delta_{0,d}$ is the only simple $d$-polytope with strictly less than $2d$ vertices;
\item the simplicial prism $\Delta_{1,d-1}$  is the only simple $d$-polytope with between $2d$ and $3d-4$ vertices;
\item $\Delta_{2,d-2}$ is the only simple $d$-polytope with  $3d-3$ vertices;
\item the only simple $d$-polytope with  $3d-2$ vertices is the 6-dimensional polytope $\Delta_{3,3}$;
\item the only simple $d$-polytopes with  $3d-1$ vertices are the polytope $J_d$, the 3-dimensional cube $\Delta_{1,1,1}$ and the 7-dimensional polytope $\Delta_{3,4}$;
\item there is a simple polytope with $3d$ vertices if, and only if, $d= 4$ or 8; the only possible examples are $\Delta_{1,1,2}$, $\Gamma_{2,2}$ and  $\Delta_{3,5}$.
\end{enumerate}
\end{lemma}

\begin{proof} Assertions (i) to (v) are simply a rewording of \cite[Lem.~10(ii)-(iii)]{PrzYos16}. We prove (vi). Recall that two simple polytopes with the same graph must be combinatorially equivalent \cite[\S3.4]{Zie95}.

For $d=4$ or 8, the validity of the three examples given is easy to verify. Conversely, suppose $P$ is a simple polytope with $3d$ vertices. Then $d$ must be even.

Every facet of $P$ is simple, and so has an even number of vertices (because $d-1$ is odd).

If there were a facet with $3d-2$ vertices, there would be $3d-2$ edges running out of it and $2(d-1)$ edges running out of the two external vertices. But $2d-2\ne3d-2$.

So any facet has at most $3d-4=3(d-1)-1$ vertices. We know from (i) to (v) that any facet must be a simplex, a prism, or have $3d-6$ or $3d-4$ vertices.

First suppose some $F$ has $3d-4$ vertices. The four vertices outside all have degree $d$, so the number of edges between them and $F$ is at least $4(d-3)$. On the other hand, there are exactly $3d-4$ edges running out of $F$. So $4(d-3)\le3d-4$ whence $d\le8$.

Next suppose some facet $F$ has $3(d-1)-3$ vertices. Each of the six vertices outside each has degree $d$, so the number of edges between them and $F$ is at least $6(d-5)$. On the other hand, there are exactly $3d-6$ edges running out of $F$. Again $d\le8$.

Now we show that the 3 examples in the first paragraph are the only possibilities when $d= 4$ or 8. If $d=8$, every facet must be $\Delta_{2,5}$ or $\Delta_{3,4}$; a short calculation shows that $P$ must be $\Delta_{3,5}$.

If $d=4$, simplicity ensures that no facet has 10 vertices. Thus every facet is either a simplex, a prism, a cube or the 5-wedge $J_3$. \cref{prop:simple-prisms-simplices} ensures that at least one facet is neither a simplex nor a prism. If some facet is a cube, it is not hard to verify that four of the other facets must be prisms, and two of them must be cubes; this is the only way the 4 vertices outside can be connected to give us the graph of a simple polytope, and the graph is that of $\Delta_{1,1,2}$. Otherwise some facet is a 5-wedge, and there are a number of cases to consider. Since each such facet contains two pentagonal faces, and each pentagonal ridge belongs to two such facets, we can find a collection of facets $W_1, W_2, \ldots, W_k= W_0$, with each $W_i$ being a 5-wedge, and $W_i\cap W_{i+1}=P_i$ being a pentagon for each $i$. Clearly there must be at least three pentagonal faces, so $k\ge3$. We cannot have $k\ge5$, or there would be too many vertices. If $k=3$, note that $P_1\cap P_3$ and $P_2\cap P_3$ are both edges. If they were the same edge (namely $P_1\cap P_2\cap P_3$), then the graph of $W_1\cup W_2\cup W_3$ would contain eleven vertices, all of degree four, which is clearly impossible. Thus  $P_1\cup P_2$ (which is contained in $W_2$)  contains at least three of the vertices in $P_3$.  Then the affine hyperplane containing $W_2$ would also contain $P_3$. This means that all 3 facets lie in the same 3-dimensional affine subspace, which is absurd. So we must have  $k=4$. The graphs of these four facets  actually determine the entire graph of $P$; it is the graph of $\Gamma_{2,2}$.

Finally we need to show that the case $d=6$ does not arise. This was first proved by Lee \cite[Example 4.4.17]{Lee81}, using the $g$-theorem, but we give an independent argument. Suppose $P$ is such a 6-polytope; then every facet must be either a simplex, a prism, $J_5$ or $\Delta_{2,3}$. By \cref{prop:simple-prisms-simplices} we can choose a facet $F$ which is not as prism. Whether this facet is  $J_5$ or $\Delta_{2,3}$, we can choose a ridge $R$ therein  which is a 4-prism. Denote by $G$ the other facet containing $R$. It is not possible for both $F$ and $G$ to be $J_5$, because then $F\cup G$ would contain more than 18 vertices. If $F$ is $J_5$ and $G$ is $\Delta_{2,3}$, then $F\cup G$ would contain  18 vertices and 49 edges, with the six vertices in $F\setminus R$ guaranteeing another six edges; but $49+6>54$, contradicting simplicity. If $F$ is $J_5$ and $G$ is a prism, then $F\cup G$ would contain  16 vertices and 44 edges, with the two vertices outside $F\cup G$ guaranteeing another 11 edges; but $44+11>54$, again contradicting simplicity. Likewise if $F$  and $G$ are both $\Delta_{2,3}$, then $F\cup G$ would contain 16 vertices and 44 edges, with the two vertices outside $F\cup G$ guaranteeing another 11 edges. Finally if $F$ is $\Delta_{2,3}$ and $G$ is a prism, then $F\cup G$ would contain  14 vertices and 39 edges, with the four vertices outside $F\cup G$ guaranteeing another 18 edges; but $39+18>54$. This exhausts all the possibilities, so there is no simple 6-polytope with 18 vertices.
\end{proof}

\section{Possible values of the excess degree}

 We will denote by $\Xi(d)$ the set of possible values of the excess of all $d$-dimensional polytopes. The main result here is that the smallest two values in $\Xi(d)$ are 0 and $d-2$; nothing in between is possible.

If $j\le d-3$, then $(d-2-j)(j+1)\ge d-2$, so \cref{lem:Non-DisjointFacets} establishes this for any polytope which is not semisimple. Accordingly, we restrict our attention to semisimple polytopes. First consider the case in which every facet is simple.

\begin{proposition}\label{prop:CharPSimple} Let $P$ be a semisimple $d$-polytope in which every facet is simple. Then $P$ is simple.
\end{proposition}
\begin{proof} This is immediate from \cref{lem:simpleNonsimpleVertex}.
\end{proof}

In addition to the assumption that every two nondisjoint facets intersect at a ridge, thanks to \cref{prop:CharPSimple}, we can now assume that our nonsimple $d$-polytope $P$ contains a nonsimple facet.

It is almost obvious that the excess degree of any proper face does not exceed the excess degree of the entire polytope. It is useful to know that this inequality is always strict.

\begin{lemma}\label{lem:facet-not-all-excess} The excess of a nonsimple polytope is strictly larger than the excess of any of its facets.
\end{lemma}
\begin{proof} Suppose otherwise. Let $F$ be a facet with equal excess to the polytope. Then, every vertex outside $F$ is simple. Take a nonsimple vertex $u$ in $F$ and neighbour $x$ of $u$ lying in $P\setminus F$. Since $x$ is simple the edge $ux$ is contained in exactly $d-1$ facets, but $u$ must be contained in at least $d+1$ facets. By \cref{lem:simpleVertexOutside}, $u$ has at least two neighbours outside $F$, implying the excess of the polytope is larger than the excess of $F$.
\end{proof}
	
We are now ready to prove our fundamental theorem.

\begin{theorem}[Excess Theorem]\label{thm:excess} Let $P$ be a  $d$-polytope. Then the smallest values in $\Xi(d)$ are 0 and $d-2$. \end{theorem}

\begin{proof}

Proceed by induction  on $d$, with the  base case $d=3$ being easy. A tetrahedron is simple, and a square pyramid has excess one.

If $P$ is a simple polytope we get excess zero, so assume that $P$ is nonsimple. By virtue of  \cref{lem:Non-DisjointFacets}  we can assume that every nondisjoint pair of facets intersect at a ridge, and by \cref{prop:CharPSimple} that there exists a nonsimple facet $F$ with excess at least $d-3$ by the inductive hypothesis.

If there is a nonsimple vertex outside $F$, we are done. If a nonsimple vertex  of $P$ is simple in $F$, we are also at home.  So we can further assume that the facet $F$ contains all the nonsimple vertices of $P$ and that each nonsimple vertex $u$ of $P$ is nonsimple in $F$ and has exactly one neighbour $x$ outside $F$. In this case, the facet $F$ would contain all the excess of the polytope, which is ruled out by \cref{lem:facet-not-all-excess}.
  \end{proof}

The existence, in every dimension, of $d$-polytopes with excess $d-2$ and $d$ has been observed in \S2. The existence of $d$-polytopes with excess $d+2$, e.g. the cyclic polytope with $d+2$ vertices, is also well known. (The characterisation of $d$-polytopes with  $d+2$ vertices \cite[\S6.1]{Gru03} gives many more examples.) In the next section, we will show that polytopes with excess $d-1$ only exist in dimensions 3 and 5.  What about higher values?
It is clear that the excess degree can be any natural number if $d=3$, and any even number if $d=4$.
We can show that the excess degree takes all possible values above $d-2$ if $d=5$ or $6$. We suspect that there are gaps in the possible values of $\xi(P)$ for $d\ge7$. For $d=7$, we can show that the excess degree takes all possible values from 7 onwards, expect perhaps 11.  Apart from the excess $d-1$, we have {so far} been unable to prove the existence of any other such gap above $d-2$. The next Theorem shows that for each dimension, the number of further gaps, if any, is finite.

\begin{lemma}\label{lem:excess-few-gaps-facets} If the $d$-polytope $P$ is a pyramid with base $F$, and $F$ has $v$ vertices, then $\xi(P)=\xi(F)+v-d$.
\end{lemma}

\begin{theorem}\label{thm:excessfewgaps} If $d$ is even, then $\Xi(d)$ contains every even integer in the interval $[d\sqrt d,\infty)$.  If $d$ is odd, then $\Xi(d)$ contains every  integer in the interval $[d\sqrt{2 d},\infty)$.
\end{theorem}

\begin{proof}
According to \cite[\S10.4]{Gru03}, for any integer $v\ge6$ and any integer $e$ in the interval $[2v+2,{v\choose2}]$ there is a 4-polytope with $v$ vertices and  $e$ edges. (A stronger assertion can be made if we exclude the values $v=6,7,8,10$, but we do not need it here.) It follows that for any $v\ge6$ and any even $\xi$ in the interval $[4,v^2-5v]$ there is a 4-polytope with $v$ vertices and excess $\xi$. Let us denote by $I(v)$ the collection of all integers in the interval $$[4+(d-4)(v-5),v^2-5v+(d-4)(v-5)]$$ which have the {\bf same parity} as the two endpoints. Applying \cref{lem:excess-few-gaps-facets} to $(d-4)$-fold pyramids, we see that for any $d>4$, $v\ge d+2$ and any  $\xi\in I(v)$, there is a $d$-polytope with $v+d-4$ vertices and excess $\xi$. We now consider three cases, depending on the value of $d$.

First consider the case when $d\ge5$ is odd.  Let $v_0$ be the smallest integer greater than  $\sqrt{2d} +3$; clearly $v_0\ge6$. For any $v\ge v_0$ it is clear that $\Xi(d)$ contains the (parity based) interval $I(v)$. It is easily seen that $v^2-5v\ge(\sqrt{2d} +3)(\sqrt{2d}-2)>2d-6$, so
$4+2(d-4)<v^2-5v+2$, whence $\min I(v+2)<\max I(v)+2$, i.e. the intervals $I(v)$ and $I(v+2)$, which have elements of the same parity, overlap. Consider separately the intervals $I(v_0), I(v_0+2),I(v_0+4)\ldots$, whose elements all have the same parity, and the intervals  $I(v_0+1), I(v_0+3),I(v_0+5)\ldots$ whose elements all have the other parity. Since $\min I(v_0+1)-1\in I(v_0)$, a moment's reflection shows that $\Xi(d)$ contains every integer, both even and odd, from $\min I(v_0+1)-1$ onwards. But $v_0<\sqrt{2d}+4$, and $4\sqrt{2d}>12$, so
$$\min I(v_0+1)-1=3+(d-4)(v_0+1-5)<3+(d-4)\sqrt{2d}
<d\sqrt{2d}-9.
$$

  Thus $\Xi(d)$ contains every  integer greater than   $d\sqrt{2d}-9.$

Now consider the case when $d\ge16$ is even. Let $v_0$ be the smallest integer greater than or equal to $\sqrt d +3$. Then $v_0\ge7$. For any $v\ge v_0$ we have
$$v(v-5)\ge(\sqrt d +3)(\sqrt d -2)=d+\sqrt{d}-6\ge d-2.$$
It follows that $4+(d-4)(v+1-5)\le v^2-5v+(d-4)(v-5)+2$, which means that  $\min I(v+1)\le\max I(v)+2$, i.e. there is no gap between the parity-based intervals $I(v)$ and $I(v+1)$. Each such interval is contained in $\Xi(d)$, which thus contains all even integers in $[4+(d-4)(v_0-5),\infty)$. On the other hand, $v_0<\sqrt d+4$, and $4\sqrt d+d\ge32$, so $$4+(d-4)(v_0-5)<4+(d-4)(\sqrt{d}-1)\le d^{1.5}-24.$$ Thus $\Xi(d)$ contains every even integer greater than or equal to $d^{1.5}-24$.

For lower dimensions, we recall that if $d\ge4$ and $\xi$ is an even number in the interval $[2d-6,2d+6]$, then there is a $d$-polytope $P$ with excess $\xi$ (and $d+3$ vertices). This is a reformulation of \cite[Lemma 14]{PinUgoYos15}; a tiny modification of the proof there shows that $P$ can be chosen to have at least one {(in fact every)} facet being a simplex. Now stacking a vertex on a simplex facet will increase the excess by precisely $d$. Thus we obtain $d$-polytopes  whose excesses are every even number in the intervals $[3d-6,3d+6]$, $[4d-6,4d+6]$, $[5d-6,5d+6]$ etc. If $d\le14$, the union of these intervals is $[2d-6,\infty)$; {clearly $2d\le d^{1.5}$, so} this completes the proof.
\end{proof}

We now note that the excess degree takes all possible values above $d-2$ if $d$ is sufficiently small. In particular, if $d\le6$, the only impossible values of $\Xi(d)$ are those excluded by  \cref{thm:excess} and parity considerations.

In case $d=6, 8$ or 10, the proof just given shows that every even integer from $2d-6$ onwards is the excess degree of some polytope; and the existence of $d$-polytopes with excess $d-2, d$ and $d+2$ has already been noted. It is clear then that every even value above $d-2$  is realised as the excess of some polytope.

In case $d=5$, this proof shows that every  integer from 7 onwards is the excess degree of some polytope; the existence of 5-polytopes with excess 3 or 5 should be clear. The triplex $M_{3,2}$ has excess degree 4, and a 3-fold pyramid over a pentagon has excess degree 6.

\section{Structure of polytopes with small excess}

It is well known (see \cref{thm:shp} below) that any simple polytope is either a simplex or  decomposable. Here we will generalise this, showing that any $d$-polytope with excess strictly less than $d$ is either a pyramid or decomposable. To do this, we first show that such polytopes have a very particular vertex structure. The nonsimple vertices all have the same degree,  and if $d\ge4$, they always form a face. We begin with some background material about decomposability.

\subsection{Decomposability of polytopes}

Recall from the introduction that a polytope $P$ is {\it (Minkowski) decomposable} if it can be written as the Minkowski sum of two polytopes, neither of which is homothetic to $P$; otherwise it is {\it indecomposable}. All polygons other than triangles are decomposable; this topic becomes more serious when $d\ge3$. We generally do not distinguish between polytopes which are  combinatorially equivalent, i.e. have isomorphic face lattices. However, a polytope can have one (geometric) realisation which is decomposable  and another realisation which is indecomposable. Smilansky \cite[p. 43]{Smi87} calls a polytope $P$ {\it combinatorially decomposable (resp. indecomposable)} if whenever $Q$ is combinatorially equivalent to $P$, then $Q$ is also decomposable (resp. indecomposable); otherwise $P$ is called {\it conditionally decomposable}. In this paper we only come across polytopes which are combinatorially decomposable, just called decomposable henceforth, or combinatorially  indecomposable polytopes, just called indecomposable henceforth. But it is important to be aware that conditionally decomposable polytopes do exist; see \cite[\S5]{Kal82}, \cite[Fig. 1]{Smi87} or \cite[Example 11]{PrzYos08} for further discussion.

A major tool for establishing decomposability is the following concept. We will say that a facet $F$ in a polytope $P$ has {\bf Shephard's property} if every vertex in $F$ has a unique neighbour outside $F$. This is not an intrinsic property of $F$, but rather of the way that $F$ sits inside $P$. This concept appears implicitly in  \cite{She63}, and leads to the following result.

\begin{theorem}\label{thm:shp} A polytope $P$ is decomposable whenever there is a facet $F$ with Shephard's property, and $P$ has at least two vertices outside $F$. In other words, if a polytope $P$ has  a facet $F$ with Shephard's property, then  it is either decomposable or a pyramid with base $F$.
\end{theorem}

This result was essentially proved by Shephard \cite[Result (15)]{She63}. He made the stronger assumption that every vertex in $F$ is simple in $P$, but the general statement does follow from his proof. Another proof   appears in \cite[Prop.~5]{PrzYos16}. We will say that a polytope is a {\it Shephard polytope} if it has at least one facet with Shephard's property. It is not hard to see that a polytope is simple if, and only if,  every facet has Shephard's property.

The following related concept will also be useful. Let us say that a facet $F$ in a polytope $P$ has {\bf Kirkman's property} if every other facet intersects $F$ at a ridge. Again this is not an intrinsic property of $F$ alone. Lee \cite[p. 110]{Lee81} calls a {\it Kirkman polytope} any polytope in which at least one facet has Kirkman's property; Klee \cite[p. 2]{Kle74} had earlier made an equivalent definition for simple polytopes. The generalisation  to non-simple polytopes in \cite{Lee81} is quite natural. Following Klee, we call a {\it super-Kirkman polytope} any polytope in which {\it every} facet has Kirkman's property. This is equivalent to the dual polytope being 2-neighbourly.

Numerous examples (e.g. the cube) show that Shephard's property does not imply Kirkman's property. In the other direction, if $P$ is a pyramid over $\Delta_{m,n}$, where $m,n\ge2$, then $P$ is a super-Kirkman polytope, but only one facet has Shephard's property.

As a common weakening of both properties, we will also say that a facet $F$ in a polytope $P$ has the {\bf weak Kirkman-Shephard property} if every other facet  is either disjoint from $F$, or intersects $F$ at a ridge.Obviously a polytope is{ semisimple} if every facet has the weak Kirkman-Shephard property. The following easy result will be useful later.

\begin{lemma}\label{lem:wksp} Let $P$ be a polytope, and $F$ a facet with the weak Kirkman-Shephard property. Then every vertex which is simple in $F$ is also simple in $P$.
\end{lemma}

\begin{proof}
This is immediate from \cref{lem:simpleNonsimpleVertex}.
\end{proof}

Turning now to indecomposability, a much bigger toolkit of sufficient conditions is available. Rather than giving a detailed survey of this topic, we will just summarize the results we need. The main concept we need is due to Kallay \cite{Kal82}. He defined a {\it geometric graph} as any graph $G$ whose vertex set $V$ is a subset of $\r^d$ for some $d$, and whose edge set $E$ is a subset of the line segments joining members of $V$. He then defined a notion of decomposability for geometric graphs, and showed \cite[Thm.~1a]{Kal82} that a polytope is decomposable if, and only if, its edge graph (or skeleton) is decomposable in his sense. The study of geometric graphs is thus a useful tool for establishing indecomposability. For a detailed proof of the next result, see {\cite{PrzYos08}} and the references therein.

\begin{theorem}\label{thm:McMullen-Decomp}
\begin{enumerate}[(i)]

\item If the vertices of a geometric graph are affinely independent, and its edges form a cycle, then the graph is indecomposable.

\item If $G_1$ and $G_2$ are two indecomposable geometric graphs with (at least) two vertices in common, then the graph $G_1\cup G_2$ is also indecomposable.

\item If $P$ is a polytope, $G$ is a subgraph of the skeleton of $P$, $G$ contains at least one vertex from every facet of $P$, and $G$ is an indecomposable graph, then $P$ is an indecomposable polytope.
    \item Conversely, if $P$ is an indecomposable polytope, then its skeleton $G$ is an indecomposable geometric graph.
\item If $G_0 = (V_0,E_0)$ is an indecomposable geometric graph, $G_1 = (V_1,E_1)$ is another geometric graph with one more vertex and two more edges, i.e. there is a unique $v\in V_1 \setminus V_0$, and distinct vertices $u$ and $w$ in $V_0$,
such that $E_1 = E_0 \cup\{[u, v],[v,w]\}$, then $G_1$ is also indecomposable.
\end{enumerate}
\end{theorem}

Note in part (v) that no assumption is made about whether $[u,w]$ is an edge of either graph. A Hamiltonian cycle through the vertices of a simplex is an affinely independent cycle; this is one way to see that a simplex is indecomposable. Application of part (ii) then shows that any simplicial polytope is indecomposable. In particular, the dual of any simple polytope must be indecomposable. The corollary below generalises that. It depends on the following result, which seems to be a basic and fundamental result on decomposability, but we could not find a reference.

\begin{theorem}\label{thm:d-1DecomposableFacets} Let $P$ be a $d$-polytope with strictly less than $d$ decomposable facets. Then $P$ is indecomposable.
\end {theorem}

\begin{proof} Let $n$ denote the number of decomposable facets of $P$, with $n<d$. Consider the corresponding collection of vertices in the dual polytope $P^*$. Thanks to Balinski's Theorem again, their removal does not disconnect $P^*$, so all the other vertices (which correspond to indecomposable facets of $P$) can be ordered into a sequence, possibly with repetition, so that any successive pair defines an edge of $P^*$.

This means that the indecomposable facets of $P$ form a strongly connected family, in the sense that each successive pair intersects in a ridge. This is much more than we need; intersecting in an edge gives us a suitably connected family of indecomposable faces. The union of this family is clearly an indecomposable graph. It touches every facet because there are only $n$ other facets, and any facet intersects at least $d$ other facets. Hence \cref{thm:McMullen-Decomp} gives the conclusion.
\end{proof}

Furthermore, this is best possible: there are decomposable $d$-polytopes with precisely $d$ decomposable facets, namely $d$-prisms and capped $d$-prisms.

\begin{corollary}\label{cor:DualsExcessD-2-D-1-Indecomp} Let $P$ be $d$-polytope at most $d-1$ nonsimple vertices.  Then the dual polytope is indecomposable.
\end{corollary}
\begin{proof}
\cref{thm:d-1DecomposableFacets} gives the result at once,  as duals of $d$-polytopes with at most $d-1$ nonsimple vertices  have at most $d-1$ decomposable facets.
\end{proof}

This is also best possible. The bipyramid over a $(d-1)$-simplex has exactly $d$ nonsimple vertices, but its dual, the prism,  is decomposable.

\subsection{Polytopes with small excess}

\cref{cor:DualsExcessD-2-D-1-Indecomp} shows the indecomposability of the duals of $d$-polytopes with at most $d-1$ nonsimple vertices.
In this section, we establish stronger conclusions for polytopes with small excess, namely $\xi=d-2$ or $d-1$. We show that this imposes strong restrictions on the distribution of nonsimple vertices. A $d$-polytope with excess $d-2$ either has a single vertex with excess $d-2$, or $d-2$ vertices each with excess one; in the latter case, the nonsimple vertices form a $(d-3)$-face. A polytope can have excess $d-1$ only if $d=3$ or 5, and cannot have three nonsimple vertices. Thus there is at least one further gap in the possible values of the excess degree, from dimension 7 onwards.

We begin by examining the intersection patterns of facets; these are also severely restricted by the assumption of low excess. With the exception of ten 3-dimensional polyhedra, every $d$-polytope with excess $d-2$ or $d-1$ is a Shephard polytope. Nine of these ten turn out to be (Minkowski) decomposable. Thus every polytope with low excess is either decomposable or a pyramid or $\TA$ (see \cref{fig:polytopes}(d)); and we will see that in all cases their duals are indecomposable. Of course a polytope with excess zero is either decomposable or a simplex.

Thus the excess theorem is highly applicable to questions about the decomposability of polytopes. The results just mentioned are best possible since there are $d$-polytopes with excess $d$ which are decomposable and $d$-polytopes that are indecomposable. For instance, the capped prism $CP_{d,d}$ is decomposable and self-dual, a bipyramid over a $(d-1)$-simplex is indecomposable with decomposable dual, and the  capped prisms $CP_{k,d}$ for $k<d$ are decomposable with indecomposable duals. For $d=3$, all polyhedra with seven vertices and seven faces have excess $d$; with the exception of the capped prism, they are all indecomposable with indecomposable duals.

For simplicity of expression, we will sometimes simply state the conclusion that we have a Shephard polytope; bear in mind that \cref{thm:shp} then guarantees that the polytope is either decomposable or a pyramid.

We begin by making explicit the following easy consequence of \cref{lem:Non-DisjointFacets}.

\begin{lemma}\label{lem:facetintersection}
Let $F$ and $G$ be distinct nondisjoint facets of a polytope $P$ and set $j=\dim(F\cap G)$. If $\xi(P)=d-2$, then $j$ is either 0,  $d-3$ or $d-2$. If $\xi(P)=d-1$, then either $j=0, d-3$ or $d-2$; or $j=1$ and $d=5$. In case $j\ne d-2$, $F\cap G$ is either a simplex or a quadrilateral.
\end{lemma}

\begin{proof}
Let $k$ be the number of vertices in $F\cap G$; then \cref{lem:Non-DisjointFacets} informs us that $(j+1)(d-2-j)\le k(d-2-j)\le\xi(P) \le d-1$.  Clearly $j=d-2$ is one possibility. Henceforth, it is enough consider only the case $j\le d-3$.

The inequality $(j+1)(d-j-2)\le d-1$ is equivalent to $j(d-3-j)\le1$. The only solutions for this are $j=0$, $j=d-3$ and $j=1=d-3$. Clearly
$F\cap G$ is a simplex if $j=0$ or 1.

If $j=d-3$, our original inequality becomes $k\le\xi(P)$. If $\xi(P)=d-2$, then $k\le j+1$ and again $F\cap G$ is a simplex. In  case $\xi(P)=d-1$, it is also possible that $k=j+2$. But then the contribution to the excess of the edges leaving $F\cap G$ is already $d-1$, meaning that $F\cap G$ must be simple. This is only possible if $j=2$, and then $k=4$ and $d=5$.
\end{proof}

This result enables us   to study the structure of  polytopes with excess degree $d-2$. For 3-dimensional polyhedra this is quite simple. Any such polyhedron has a unique nonsimple vertex, which has degree four, and so the polyhedron is either decomposable, or a quadrilateral pyramid. We now proceed to higher dimensions.

\begin{lemma}\label{lem:Excessd-2Facets0}  Let $P$ be $d$-polytope with excess degree $d-2$.  If $P$ has a unique nonsimple vertex (necessarily with excess $d-2$), then it has two nondisjoint facets intersecting at just this vertex. Conversely, if $F_1$ and $F_2$ are any two nondisjoint facets of $P$ such that $\dim F_1\cap F_2= 0$, then this vertex is the only nonsimple vertex (with excess $d-2$), and consequently $P$ is a Shephard polytope. Furthermore, $P$ has a  facet with excess $d-3$.
\end{lemma}

\begin{proof}
Suppose there is a unique nonsimple vertex $u$ with degree $2(d-1)$. Thanks to \cref{lem:vertex-all-neighbours-simple}, the vertex figure $P/u$ of $P$ at $u$ is a simple $(d-1)$-polytope, in fact a simplicial prism. Let $R_1$ and $R_2$ be the two opposite simplices in $P/u$. Then the two facets containing $u$ arising from $R_1$ and $R_2$  intersect only at $u$.

Conversely, suppose $\{u\}= F_1\cap F_2$. Then $u$ has degree $d-1$ in both facets, and excess $d-2$. Thus every other vertex in $P$  is simple. There is a facet in $P$ which does not contain $u$, which clearly has Shephard's property.

Finally,  let $F$ be the other facet arising from a ridge of $F_1$ which contains $u$. The facet $F$ must then intersect $F_2$ at a ridge; it can't intersect $F_2$ at  a $(d-3)$-face since $u$ is the only nonsimple vertex (see \cref{lem:Non-DisjointFacets}). The degree of $u$ in $F$ is $2(d-2)=d-1+d-3$.
\end{proof}

\begin{lemma}\label{lem:Excessd-2Facetsd-3}  Let $P$ be $d$-polytope with excess degree $d-2$.  Let $F_1$ and $F_2$ be any two facets of $P$ such that $K= F_1\cap F_2$ has dimension $d-3$. Then
\begin{enumerate}[(i)]
\item
$K$ is a simplex, and each of its $d-2$ vertices has excess degree  one in $P$.

\item $P$  is a Shephard polytope. If it is not decomposable, then it is a $(d-2)$-fold pyramid.

\item Every facet in $P$ intersecting $K$ but not containing it misses exactly one vertex of $K$ and every vertex of $K$ in the facet has degree $d$; thus the facet has excess $d-3$.

\item There are precisely four facets containing $K$, and each of them is simple.
\end{enumerate}
\end{lemma}

\begin{proof} The result is  clearly true for $d=3$, so assume $d\ge4$. From \cref{lem:Non-DisjointFacets} it follows that every vertex in $K$ has degree $d+1$ in $P$, that $K$ is a simplex, and that every other vertex in $P$ is simple. So (i) is proved. Moreover  every vertex in $K$ is simple in both $F_1$ and $F_2$, and has no neighbours outside $F_1\cup F_2$.

Since $d\ge 4$, we can find distinct vertices $u,v\in K$. Consider a facet $F$ containing $u$ but not $v$; we claim that any such facet has Shephard's property. Note that $F$ intersects each of $F_1$ and $F_2$ at a ridge, and contains every vertex in $K$ but $v$, every neighbour of $u$
 in $F_1\setminus K$, and every neighbour of $u$ in $F_2\setminus K$. The same argument applies to every vertex in $F\cap K$. This in turn implies that every vertex in $F$ has exactly one neighbour outside $F$. In other words, the $d-3$ vertices in $F\cap K$ has degree $d$ in $F$, and hence the facet $F$ has excess $d-3$. This completes the proof of (iii).

Thus, by \cref{thm:shp} $P$ is decomposable unless $\{v\}=P\setminus F$, in which case $P$ is a pyramid over $F$ with apex $v$. As there is nothing special about $u$ or $v$, in the case of $P$ being indecomposable, we see that $P$ is a pyramid with every vertex in $K$ acting as an apex, that is, $P$ is a $(d-2)$-fold pyramid. This  proves (ii).

For (iv), fix  a vertex $v\in K$, let $x_1$ and $x_2$ be the neighbours of $v$ in $F_1\setminus K$, and let $x_3$ and $x_4$ be the neighbours of $v$ in $F_2\setminus K$. For each $i$, denote by $R_i$ the smallest face containing $K\cup\{x_i\}$. Since $K$ has codimension two in $F_1$, it must be the intersection of two ridges in $F_1$ (i.e. facets of $F_1$). Thus $R_1$ and $R_2$ must be ridges of $P$. Likewise $R_3$ and $R_4$ must be ridges of $P$. These four ridges are distinct, so there must be at least four facets of $P$ which contain $K$. We need to show that there are only four, not five or six.

Denote by $F_{i,j}$ the smallest face of $P$ containing $K\cup\{x_i,x_j\}$ for each $i,j$. Obviously any facet containing $K$ must be of this form, and $F_1=F_{1,2}$ and $F_2=F_{3,4}$. With respect to the inclusion $R_1\subset F_{1,2}$, the other facet containing $R_1$ must be either $F_{1,3}$ or $F_{1,4}$; without loss of generality suppose it is  $F_{1,3}$. Since $R_1=F_{1,2}\cap F_{1,4}$ also, the unique representation of ridges as the intersection of two facets implies that $F_{1,4}$ cannot be a facet. Continuing, we conclude that $F_{1,4}=F_{2,3}=P$ and that $F_{2,4}$ is the fourth facet containing $K$. Since $F_{1,3}\cap F_{2,4}=K$, both these facets are also simple.
\end{proof}

We can now illuminate the structure of polytopes with excess $d-2$.

\begin{lemma}\label{lem:excess-d-2-4dim} Let $P$ be a 4-polytope with excess degree two. Then $P$ is not a super-Kirkman polytope.
\end{lemma}
\begin{proof}
 If there is a unique nonsimple vertex, this follows from \cref{lem:Excessd-2Facets0}.

Otherwise, $P$ has two vertices, say $u$ and $v$, of degree five. They must be connected by an edge; otherwise, the vertex figure of each would be a simple 3-polytope with five vertices by \cref{lem:face-all-neighbours-simple}. Consider a facet $F$ containing the edge $uv$. \cref{lem:facet-not-all-excess} ensures that the excess of $F$ is at most one, which means that either $u$ or $v$ must be simple in $F$. Then $F$ fails Kirkman's property by \cref{lem:simpleNonsimpleVertex}.
\end{proof}

In the next result, examples for case (i) include $B_4$, $\Sigma_d$ and $M_{d-1,1}$.
Examples for case (ii) include $A_4$, $C_d$, $M_{2,d-2}$ and pentasms.

\begin{theorem}\label{thm:excess-d-2-full-story} Any $d$-polytope $P$  with excess exactly $d-2$ either

\begin{enumerate}[(i)]
\item has a unique nonsimple vertex, which is the intersection of two facets, or
\item has $d-2$ vertices of excess degree one, which form a $(d-3)$-simplex which is the intersection of two facets.
\end{enumerate}

In either case, the two intersecting facets are both simple polytopes, $P$ is a Shephard polytope and has another facet with excess $d-3$.
\end{theorem}

\begin{proof} We proceed by induction on the dimension. The case $d=3$ is obvious, with (i) and (ii) coinciding. The case $d=4$ is settled by \cref{lem:facetintersection} and \cref{lem:excess-d-2-4dim}.

Consequently, assume henceforth that $d\ge5$.  \cref{lem:Excessd-2Facets0} and \cref{lem:Excessd-2Facetsd-3} establish the conclusion if there are two facets which intersect in a vertex or a $(d-3)$-face. So we may assume that $P$ is a super-Kirkman polytope; we will show that this case actually does not arise.

By \cref{prop:CharPSimple,thm:excess}, we can  assume that there exists a facet $F$ with excess at least $d-3$, and by \cref{lem:facet-not-all-excess} that $F$ (and every nonsimple facet) has excess exactly $d-3$.

By induction, we may suppose that $F$ satisfies either (i) or (ii).

First suppose that the facet $F$ contains a nonsimple vertex $v$, with degree $2d-4$ therein, and two ridges $R$ and $S$ such that $R\cap S=\{v\}$. Let $G$ be the other facet corresponding to $R$. Then $v$ is not simple in $G$ because it is not simple in $P$ (recall \cref{lem:facets-intersecting-at-ridges}). Consequently, $v$ must have degree $d$ in $G$, degree $2d-4+2$ in $P$ and is the only nonsimple vertex of $P$. This in turn implies that the excess of $G$ is $1=d-3$, a contradiction.

Next suppose that $F$ contains two ridges $R$ and $S$ such that $R\cap S=K$ is a $(d-4)$-simplex, with every vertex having excess one in $F$. If there were no other nonsimple vertex in $P$, then every vertex in $K$ would be adjacent to some simple vertex in $P\setminus F$, and hence to at least two vertices outside $F$ (by \cref{lem:simpleVertexOutside}), meaning the excess degree of $P$ would be at least $2(d-3)>d-2$. So there must another nonsimple vertex $u$ in $P\setminus F$, with excess one, and it will be the unique neighbour outside $F$ of every vertex in $K$. This implies that the other facet $G$ corresponding to $R$ must contain $u$. But then $G$ contains $d-2$ nonsimple vertices, each having excess degree one in $G$, which is impossible by \cref{lem:facet-not-all-excess}.
\end{proof}
	
We now turn our attention to the decomposability of $d$-polytopes with excess $d-1$. Note that they can exist only if $d$ is odd.

\begin{rmk}\label{rmk:nonpyramidalIndecomp3-polytopes}
Since 3-dimensional polyhedra with excess degree $d-1$ do not behave as neatly as those in higher dimensions, we will examine this case in detail first. Catalogues of 3-dimensional polyhedra help us to clarify the situation. Numbers in (iv) in the next statement refer to the catalogue of Federico \cite{Fed74}.
The  tetragonal antiwedge has excess $d-1=2$, yet is neither a pyramid nor decomposable; we will see shortly that it is the only such example. It has two non-simple vertices, and every facet contains at least one of them, so it also fails Shephard's property.
\end{rmk}

\begin{lemma}\label{lem:dim3Excess2} Let $P$ be $3$-polytope whose  excess degree is two.   Then either
\begin{enumerate}[(i)]
\item $P$ is a pentagonal pyramid, which is indecomposable, or
\item $P$ is $\TA$, which is also indecomposable but not a pyramid, or
\item $P$ has Shephard's property but is not a pyramid, and hence is decomposable, or
\item $P$ is one of $F27$, $F31$, $F37$, $F105$, $F109$, $F140$, $F159$, $F163$, $F172$, which are all decomposable, but are not Shephard polytopes.
\end{enumerate}
\end{lemma}

\begin{proof}
Excess two means that for some $k$, $P$ has $2k$ vertices and $3k+1$ edges. Either one vertex has  degree five, or two vertices have  degree four. In either case, the number of faces containing a non-simple vertex is at most eight. So if $k\ge6$, then $P$ has  $k+3\ge9$ faces, one of which contains only simple vertices, and so has Shephard's property.
If a pyramid has excess two, its base must be a pentagon.

If $k=5$ or 4, then $2k>k+3$, and $P$ must be decomposable by \cite[Theorem 6.11(a)]{Smi87}. The examples which fail Shephard's property can be verified from catalogues; they are as listed.

If $k=3$, then $P$ is $\TA$ or a pentagonal pyramid.
\end{proof}

We have seen that for a polytope with excess $d-2$, the non-simple vertices form a face. We will see that this is also true for polytopes with excess $d-1$, provided $d>3$. For $d=3$, it is false; in every example in case (iv) above, the two nonsimple vertices are not adjacent. Truncating simple vertices then yields many more examples which fall into case (iii).

\begin{lemma}\label{lem:Excessd-1Facets0}  Let $P$ be $d$-polytope with   excess degree  $d-1$, and two facets  $F_1$ and $F_2$ intersecting at a vertex.  Then $d=3$, and the polytope is as described in \cref{lem:dim3Excess2}.
\end{lemma}

\begin{proof}
Let $\{u\}= F_1\cap F_2$. We consider two cases: either $u$ is the only nonsimple vertex in $P$ or there is another nonsimple vertex $v$ in $P$.

In the former case, every vertex in $P$ other than $u$ is simple and the vertex figure of $P$ at $u$ is a simple $(d-1)$-polytope with $2d-1$ vertices (recall \cref{lem:vertex-all-neighbours-simple}), so \cref{lem:CountSimpleP}(ii) forces $2d-1\ge3(d-1)-3$, whence $d=3$ or 5.  Assume $d=5$. The vertex figure $P/u$  has $2d-1=3(d-1)-3$ vertices, and by \cref{lem:CountSimpleP}(iii) can only be $\Delta_{2,2}$ (cf.~\cref{fig:polytopes}(c)), which has the property that any two 3-faces intersect at a 2-face. Thus, every two facets in $P$ containing $u$ must intersect at a ridge, a contradiction.

In the latter case, the vertex $u$ has degree $2d-2$ and $v$ has degree exactly $d+1$.  For $d>3$, \cref{lem:vertex-all-neighbours-simple} applied to $v$  ensures that $u$ and $v$ are adjacent. Without loss of generality, assume $v\in  F_1$. Let $F$ be any other facet in $P$ containing $v$ but not $u$.  If the intersection of $F$ and $F_1$ is not a ridge, then  \cref{lem:Non-DisjointFacets} tell us that every vertex in $F\cap F_1$ is nonsimple, meaning $v$ is the only element therein. Then $v$ has excess $d-2$ in $P$, which is impossible when $d>3$. So  every facet containing $v$ but not $u$ must intersect $F_1$ at a ridge.  \cref{lem:simpleNonsimpleVertex}  in turn implies that $v$ is nonsimple in $F_1$.

Then \cref{thm:excess} gives $1=\xi(F)\ge d-3$, again forcing $d=3$ (recall $d$ is odd).

{Thus, $d=3$ in both cases. Finally} note that if a 3-polytope is nonsimple, then a nonsimple vertex must lie in the intersection of a pair of facets, so all such examples are described in \cref{lem:dim3Excess2}.
\end{proof}

We now deal with one particular situation which is mentioned in \cref{lem:facetintersection}.

\begin{lemma}\label{lem:Excess4Facets1}  Let $P$ be $5$-polytope with  excess degree four.  Let $F_1$ and $F_2$ be any two facets of $P$ such that $\dim F_1\cap F_2= 1$. Then $P$ is a Shephard polytope.
\end{lemma}
\begin{proof} Let $F_1\cap F_2=[u,v]$. Then $u$ and $v$ are both simple in both $F_1$ and $F_2$, and every other vertex in the polytope is simple. Consider a facet $F$ containing $u$ but not $v$. Then, in view of \cref{lem:facetintersection} and \cref{lem:Excessd-1Facets0} we can assume that $F$ intersects each $F_i$  at a ridge for each $i$. This implies that $u$ has exactly one neighbour outside $F$. Since every other vertex in $F$ is simple, $F$ has Shephard's property, as required.
\end{proof}
		
The next two lemmas deal with the case of two nondisjoint facets intersecting at a $(d-3)$-face or only at ridges.

\begin{lemma}\label{lem:Excessd-1Facetsd-3}  Let $P$ be $d$-polytope with excess degree of $d-1$, with  $d>3$.  Let $F_1$ and $F_2$ be any two facets of $P$ whose intersection is a subridge. Then  $d=5$,  {the face $ F_1\cap F_2$ is a quadrilateral} and $P$ is a Shephard polytope.
\end{lemma}

\begin{proof} Recall that $d$ must be odd. Let $K$ denote the subridge $F_1\cap F_2$. We consider two cases: either (1) the excess degree of $P$ comes only from the nonsimple vertices in $K$, or (2) the excess degree of $P$ comes from the nonsimple vertices in $K$ and a nonsimple vertex $v$ outside $K$. Note that, in the first case, $K$ must have $d-2$ vertices, unless $d=5$, in which case it could have $d-1$ vertices. Indeed, suppose $K$ has $d-1$ vertices. From \cref{lem:Non-DisjointFacets} it  follows that every vertex in $K$  has degree $d+1$ in $P$ and that $K$ is a simple polytope, but the last assertion is only possible if $K$ is 2-dimensional, i.e. $d-3=2$. So  in this particular case $d= 5$ and $K$ is a quadrilateral.

{\bf Case 1.} The excess degree of $P$ comes only from the (nonsimple) vertices in $K$.

We will first rule out this case when $K$ has $d-2$ vertices. Then it is a simplex, and  there exists  a vertex in $K$, say $u$, with degree $d+2$ in $P$, while every other vertex in $K$ has degree $d+1$ in $P$ and is simple in both $F_1$ and $F_2$. If $u$ were nonsimple in $F_1$ or $F_2$, then $1=\xi(F_1)\ge d-3$ or $1=\xi(F_2)\ge d-3$ by the Excess Theorem. So $u$ is simple in $F_1$ and $F_2$. Denote by $v$ the unique neighbour of $u$ outside $F_1\cup F_2$.

We consider the set $\mathcal{F}_u$ of facets containing $u$. Since $u$ is simple in $F_1$, there are exactly $d-1$ facets in $\mathcal{F}_u$ intersecting $F_1$ at a ridge. Out of these $d-1$ facets, there are $d-3$ facets which miss one vertex in $K$; consider one such facet $F$. The facet $F$ must also intersect $F_2$ at a ridge, otherwise the excess of $P$ would be much larger. Consequently, every vertex in $(F\cap K)\setminus \{u\}$ (of which there are $d-4$) has excess one in $F$, i.e. has $d$ neighbours in $F$. If $v$ belonged to $F$, then $u$ would have $d+1$ neighbours in $F$ and $\xi(F)=d-4+2$. However, since $d-1$ is even, $\xi(F)$ cannot be odd. Thus, $v\not \in F$.

Suppose there is a facet $G$ in $\mathcal{F}_u$ which intersects $F_1$ at $K$; by \cref{lem:Excessd-1Facets0} the dimension of any such intersection $F_1\cap G$ cannot be smaller than $d-3$. For any vertex in $K$ other than $u$ has precisely two neighbours in $F_2\setminus K$,  both of which then must be in $G$. This means that $G$ must be $F_2$.  Consequently, there are exactly $d+1$ facets in $\mathcal{F}_u$: $F_1$, $d-1$ facets intersecting $F_1$ at a ridge, and $F_2$.  Hence there are at most two facets in $\mathcal{F}_u$ containing the vertex $v$. But there must be exactly $d-1$ facets in $\mathcal{F}_u$ containing $v$. So $d-1\le 2$.

For the particular case of $d=5$ and $K$ having $d-1$ vertices, every vertex in $K$ has excess degree one, so we can choose $u$ to be any of them.  Let $w\ne u$ be a vertex in $K$  (this step requires $d>3$). Consider a facet $F$ containing $w$ but not $u$.  Then from \cref{lem:facetintersection} and \cref{lem:Excessd-1Facets0} it ensues that $F$ intersects each $F_1$ and $F_2$ at a ridge. Thus $F$ contains every neighbour of $w$ in $F_1\setminus K$ and $F_2\setminus K$. The same argument applies to every vertex $\ne u$ in $F\cap K$. This in turn implies that  $F$ has Shephard's property.

We next rule out the second case for $d>3$.

{\bf Case 2.} Every vertex in $K$ has degree $d+1$ in $P$ and is simple in both $F_1$ and $F_2$, and there is a nonsimple vertex $v\not \in K$ of degree $d+1$.

If all the neighbours of $v$ are simple in $P$, \cref{lem:vertex-all-neighbours-simple}  ensures that the vertex figure of $P$ at $v$ is a simple $(d-1)$-polytope with $d-1+2$ vertices, which implies $d=2$ or 3 by \cref{lem:CountSimpleP}. So $v$  must be adjacent to at least one vertex $u\in \ver K$, and thus belongs to either $F_1\setminus K$ or $F_2\setminus K$, say $F_1\setminus K$.

We claim that every vertex $w\ne u$ in $K$ (which do exist since $d>3$) must be adjacent to $v$. Suppose otherwise. \cref{rmk:two-vertices} gives us a facet $F$ containing $v$ but not $u$. Since no vertex has excess degree two or more, $F\cap F_1$ cannot be a vertex or an edge. Suppose $F\cap F_1$ is a subridge; then it is a simplex and must contain $v$ and all $d-3$ vertices from $K\setminus\{u\}$. Clearly then, $v$ is adjacent to all $d-3$ such vertices.
The remaining possibility is that $F$  intersects $F_1$ at a ridge. Since $F_1$ cannot have excess degree precisely one, $v$ must be simple in $F_1$.
Thus there is only one such facet $F$, thereby implying that there are at least $d$ facets containing the edge $uv$. The edge $uv$ is contained in exactly $d-2$ $(d-2)$-dimensional faces of $F_1$ (which are of courses ridges in $P$). Thus there are at least two facets $G_1$ and $G_2$ containing the edge $uv$ intersecting $F_1$ at a $(d-3)$-face, which is a simplex.
Since $w\not \in G_i$ ($i=1,2$), each $G_i$ misses exactly one vertex of $K$, namely $w$. This in turn  implies that $G_1\cap F_1=G_2\cap F_1=(K\setminus\{w\})\cup \{v\}=:K'$. But then the intersection of each $G_i$ with $F_2$ must necessarily be the unique ridge of $F_2$ containing $K\setminus\{w\}$, a contradiction. Let $R$ be the ridge (of $P$) in $F_1$ containing $K\cup \{v\}$. Since $v$ is adjacent to every vertex in $K$ and since $R$ is a simple polytope, $R$ must be a $(d-2)$-simplex.

Consider the other facet $F_3$ containing the ridge $R$. Note that every vertex in $K$ is simple in $F_3$, and thus, $v$ must be simple in $F_3$. Let $e_v$ denote the unique edge incident with $v$ not contained in $F_1\cup F_3$ and let $\mathcal{F}_e$ denote the set of facets containing $e$. Any facet $F\in\mathcal{F}_e$  must intersect $F_1$ at a $(d-3)$-face or at a ridge; in any case, it must miss precisely one vertex of $K$, say $w_F$.   In this case, $F$ intersects $F_2$ at a ridge, since $d>3$ and $F$ cannot intersect $F_2$ at a $(d-3)$-face. It thus follows that for each $w_F$ there exists exactly one such facet $F$. But there are only $d-2$ vertices in $K$ and we need $d-1$ facets in $\mathcal{F}_e$. This contradiction completes the proof of the lemma.
\end{proof}

\begin{lemma}\label{lem:degree-d+1-ridges} Any semisimple $d$-polytope, in which every vertex has degree at most $d+1$, is actually simple.
\end{lemma}

\begin{proof} The case $d\le4$ follows from \cref{lem:semilowdim}. We will show first that  any  facet $F$ is also semisimple. The result then follows by induction on $d$. We may assume now that $d\ge5$.

Suppose otherwise; that is there are two $(d-2)$-faces $R_1$ and $R_2$ in $F$ intersecting at a $(d-4)$-face $J$. {(If $J$ had lower dimension, every vertex in $J$ would have excess at least two in $F$.) Choose any $u \in J$. Then $u$ has excess one in $F$ and so is  nonsimple  in $P$;  \cref{lem:simpleNonsimpleVertex}  ensures that $u$
}
is nonsimple in any facet containing it. Denoting by $F_i$ the other facet containing $R_i$ ($i=1,2$), we see that $u$ is nonsimple in $F$, $F_1$ and $F_2$,  and it has only one neighbour outside $F$, so $u$ must be nonsimple in each $R_i$. But then, the degree of $u$ in $F$ is the sum of the degree of $u$ in $R_1$ ($d-1$) plus the edges of $u\in R_2\setminus J$ (two), a contradiction.
\end{proof}

\begin{lemma}\label{lem:excess-d-1-ridges} Let $P$ be a semisimple $d$-polytope with excess $d-1$. Then $d= 5$, and there is a unique nonsimple vertex in $P$, whose vertex figure is combinatorially equivalent to $\Delta_{2,2}$.  In particular, $P$ is  a Shephard polytope.
\end{lemma}
\begin{proof} Recall that the excess can equal $d-1$ only if $d$ is odd. Note that $d>3$, since \cref{lem:semilowdim} gives at once that all 3-dimensional semisimple polytopes are simple. By \cref{prop:CharPSimple,thm:excess}, we know that there exists a facet  with excess at least $d-3$, and by \cref{lem:facet-not-all-excess} that this facet (and every nonsimple facet) has excess exactly $d-3$; note that since $d-1$ is even, the  excess of a facet cannot be $d-2$. It follows that every nonsimple ridge has excess $d-4$. We also know that any such nonsimple facet $F$ contains  either   $d-3$ nonsimple vertices, each with excess one in $F$, or has a unique nonsimple vertex, with excess $d-3$ in $F$. In the first case, $P$ has either $d-3, d-2$ or $d-1$ nonsimple vertices. In the second case, $P$ has at most three nonsimple vertices.

We begin with the important observation that when a facet $F$ contains all the nonsimple vertices of $P$, then $P$ contains at most two nonsimple vertices. From \cref{lem:simpleVertexOutside} it follows that every nonsimple vertex in $F$ is adjacent to at least two vertices outside $F$. The conclusion is now clear, since $F$ has excess $d-3$ and $P$ has excess $d-1$.

We consider first the case that $P$ contains a unique nonsimple vertex $u$; in fact, this is the only case which actually occurs. The vertex figure $P/u$ of $P$ at $u$ is a simple $(d-1)$-polytope with $2d-1$ vertices by \cref{lem:vertex-all-neighbours-simple}), and then \cref{lem:CountSimpleP}(ii) forces $d\le5$. Since the case $d=3$ has been excluded, we have $d=5$, and the vertex figure $P/u$  has $2d-1\ge3(d-1)-3$ vertices  by virtue of \cref{lem:CountSimpleP}(ii). By \cref{lem:CountSimpleP}(iii), $P/u$ is $\Delta_{2,2}$, which is depicted in \cref{fig:polytopes}(c).

Henceforth we assume that $P$ has at least two nonsimple vertices. In this case,  any nonsimple vertex must be adjacent to at least one other nonsimple vertex. If there are only two nonsimple vertices, and they are not adjacent, one of them will have excess at most $\h(d-1)$, and truncating it will give us a new simple facet with strictly less than $\h(3d-1)$ vertices. If there are three or more nonsimple vertices, all of them will have excess at most $d-3$, so truncating one which is adjacent only to simple vertices would give us a new simple facet with strictly less than $2(d-1)$ vertices. These are all impossible by \cref{lem:CountSimpleP}, if $d>3$.

If $P$ contains exactly two nonsimple vertices, say $u$ and $v$, they must be connected by an edge, and hence contained in some facet $F$.
According to \cref{thm:excess-d-2-full-story}, $F$ contains exactly $d-3$ vertices with excess one. So  again $d-3=2$.
Consider a facet $F_u$ containing $u$ but not $v$, and a facet $F_v$ containing $v$ but not $u$ (recall ~\cref{rmk:two-vertices}).
By \cref{lem:facets-intersecting-at-ridges}, $F_u$ is not simple and so has excess $d-3=2$, every vertex in $F_u$ sends exactly one edge outside  $F_u$. The vertex figure $F_u/u$ of $F_u$ at $u$ is a simple 3-polytope with six vertices; that is, it is a simplicial 3-prism. This implies that the two 3-faces $R_1$ and $R_2$ of $F_u$ arising from the opposite simplices of $F_u/u$  intersect only at $u$. Consequently, the intersection of the other two facets of $P$ containing the ridges $R_1$ and $R_2$ has dimension at most one, a contradicting semisimplicity.

Now consider the case that $P$ contains exactly three nonsimple vertices, $v_1,v_2$ and $v_3$. Without loss of generality, we can assume that  $v_2$ is adjacent to both $v_1$. Choose any facet $F$ containing the edge joining them, and note that $F$ has excess $d-3$ and does not contain $v_3$. Thanks to \cref{thm:excess-d-2-full-story}, we have $d-3=2$, each $v_i$ has excess one in the appropriate facet, and the total excess is four. So two of the nonsimple vertices have degree six in $P$, and one has degree seven.   Without loss of generality, we can assume that $v_1$ has degree six, and \cref{lem:simpleVertexOutside} tells us that it is adjacent to only one vertex outside $F$, which must be nonsimple. Thus $v_1$ is also adjacent to $v_3$. Now $v_1$ has exactly six neighbouring vertices, and every facet containing it is nonsimple. It follows that $v_1$ is contained in exactly six facets, one for each set of five of its neighbours. In particular, there is a facet (in fact there are four such facets) containing $v_1,v_2$ and $v_3$, contrary to our previous observation.

If $P$ contains exactly $d-3$ nonsimple vertices, they must all belong to the same facet, a situation whose impossibility has already been demonstrated.

We now assume that $P$ contains exactly $d-2$ nonsimple vertices. Then $d-3$ of them will constitute a simplex subridge $K$, contained in a facet $F$ with excess $d-3$, and at least $d-4$ of them will have excess one in $P$. Denote by $w$ the nonsimple vertex outside $F$. Each of these $d-4$ vertices will be adjacent to a unique vertex outside $F$, which must be nonsimple, i.e. they are all adjacent to $w$. Now consider a ridge $R$ such that $K\subset R\subset F$, and let $G$ be the other facet containing $R$. Then $G$ must contain $w$ and $K$. But we have already seen that no facet can contain every nonsimple vertex.

Finally, if $P$ contains exactly $d-1$ nonsimple vertices, they must all have excess 1, and the impossibility of this is given by \cref{lem:degree-d+1-ridges}.
\end{proof}

\begin{theorem}\label{thm:excessdecomp} Let $P$ be $d$-polytope with  excess degree  $d-1$, where  $d>3$. Then $d=5$ and either

\begin{enumerate}[(i)]
\item there is a single vertex with excess four, {which is the intersection of three facets}, and whose vertex figure is $\Delta_{2,2}$; or
\item there are two vertices with excess two, the edge joining them is the intersection of two facets and its underfacet is either $\Delta_{1,1,2}$ or $\Gamma_{2,2}$; or
\item there are four vertices each with excess one, which form a quadrilateral 2-face which is the intersection of two facets, and whose underfacet is the tesseract  $\Delta_{1,1,1,1}$.
\end{enumerate}

In all cases $P$ is a Shephard polytope.
\end{theorem}

\begin{proof}  By \cref{lem:facetintersection} any two nondisjoint facets in $P$ intersect either at a vertex, an edge with $d=5$, a subridge, or a ridge. The first possibility is excluded by \cref{lem:Excessd-1Facets0}. If there are two facets which intersect in an edge, and $d=5$, then the underfacet of the edge must be a simple 4-polytope with 12 vertices, so by \cref{lem:CountSimpleP}(vi) it is either $\Delta_{1,1,2}$ or $\Gamma_{2,2}$. This is case (ii).

If two facets intersect in a subridge $K$, \cref{lem:Excessd-1Facetsd-3} ensures that $d=5$ and $K$ is a quadrilateral face. Let $G$ be an underfacet of $K$,  replace $P$ by the polytope which is the convex hull of $G$ and $K$, and denote by $F_1$ and $F_2$ the corresponding facets in this smaller polytope. Each of the four vertices in $K$ is incident with two edges in $F_1$ and with two edges in $F_2$. Thus both ridges $R_i=F_i\cap G$ have eight vertices, and are simple, so each must be either a cube or a 5-wedge. However each $F_i$ is the convex hull of $R_i\cup K$, and is simple, so by \cref{lem:CountSimpleP}(vi) again each is either $\Delta_{1,1,2}$ or $\Gamma_{2,2}$. If we remove an quadrilateral face from a copy of $\Gamma_{2,2}$, the resulting graph is not the graph of a polytope. So for each $i$, $F_i$ is a copy of   $\Delta_{1,1,2}$, and $R_i$ is a cube. Then $G$, being the convex hull of $R_1$ and $R_2$, and being simple, must be a tesseract. This is Case (iii).

Finally, if every pair of facets intersects in a ridge or the empty set, \cref{lem:excess-d-1-ridges} guarantees that we are in Case (i).
\end{proof}

Examples for cases (i), (ii) and (iii) are the pyramid over $\Delta_{2,2}$;  $M_{3,2}$ and  $B_{5}$; and  $A_{5}$ respectively. All of these have simple vertices, so repeated truncation leads to more examples.

\section{Characterisation of decomposable $d$-polytopes with $2d+1$ vertices}

The following known result, see \cite[Theorem 7.1, page 39]{Kal79} or \cite[Theorem 9]{PrzYos16}, motivates this section.

\begin{proposition}\label{prop:2ddecomp}
The only decomposable $d$-polytope with $2d$ or fewer vertices is the $d$-prism.
\end{proposition}

We extend this here by characterising all the decomposable $d$-polytopes with $2d+1$ vertices. Characterising decomposable $d$-polytopes with $2d+2$ vertices appears to be a much harder exercise. When $d=3$, there are already 11 examples, namely those in \cite[245-255]{BriDun73}. For $d=4$, a brief discussion of this problem is given in the next section. We will make repeated use of \cref{thm:McMullen-Decomp}.

\begin{rmk}\label{rmk:cappedprismstructure} It is worth recalling the structure of capped prisms and pentasms. A $d$-dimensional capped prism  is the convex hull of a simplicial prism and a single extra vertex, say $v_0$,  which lies beyond one of the simplex facets (and beneath all the other facets). Also recall that  $CP_{k,d}$ denotes the capped $d$-prism where $k$ is the minimum dimension of any face of the simplicial prism whose affine hull contains the {\it extra} vertex. If $k=1$, the capped prism will be (combinatorially) just another prism. If $k=2$, then $P$ is a pentasm, with $d^2+d-1$ edges. For $k\ge3$, we can label the vertices as $u_1,\ldots,u_d$, $v_0,\ldots, v_d$ in such a way that the edges, $d^2+d$ in total, are $[u_i,u_j]$ for all $i,j$, $[v_i,v_j]$ for all $i,j$, and  $[u_i,v_i]$ for all $1\le i\le d$.
\end{rmk}

\begin{rmk}\label{rmk:}
In the case of a pentasm, one edge, say $[v_1,v_2]$ will be absent, and then $u_1,u_2,v_0,v_1,v_2$ will form a pentagonal face. In the case of a prism, the vertex $v_0$ will be absent.
\end{rmk}

In \cite{PinUgoYos15}, we studied the minimum number of edges of polytopes with up to $2d+1$ vertices. In this section, we need only the result for $2d+1$ vertices. The corresponding result for polytopes with up to $2d$ vertices is given in the Appendix.

\begin{theorem}[{\cite[Thm.~13]{PinUgoYos15}}]\label{thm:pentasm} The polytopes with $2d+1$ vertices and $d^2+d-1$ or fewer edges are as follows.
\begin{enumerate}[(i)]
\item For $d=3$, there are exactly two polyhedra with seven vertices and eleven edges; the
pentasm, and $\Sigma_3$. None have fewer edges.

\item For $d=4$, a sum of two
triangles $\Delta_{2,2}$ is the unique polytope with 18 edges, and the pentasm is the unique polytope with 19 edges.  None have fewer edges.

\item For $d\ge5$,  the pentasm is the unique polytope with $d^2+d-1$ edges. None have fewer edges.
\end{enumerate}
\end{theorem}

\begin{lemma}\label{twoout} Let $P$ be a decomposable $d$-polytope, $F$ a facet of $P$, and suppose that there are only two vertices of $P$ outside $F$.  Then  $F$ is decomposable, and has Shephard's property.
\end{lemma}

\begin{proof} The two vertices outside $P$ are not enough to form a facet. Thus $F$ touches every facet, and so must be decomposable by \cref{thm:McMullen-Decomp}. Moreover if some vertex $v\in F$ were adjacent to both vertices $x,y$ outside $F$, then the triangle $vxy$ would be an indecomposable face touching every facet, contrary to the decomposability of $P$.
\end{proof}

\begin{lemma}\label{lowdim}
\begin{enumerate}[(i)]
\item $\Sigma_3$ is not a facet of any decomposable 4-polytope with nine vertices.
\item $\Delta_{2,2}$ is not a facet of any decomposable 5-polytope with eleven vertices.
\end{enumerate}
\end{lemma}

\begin{proof}
(i) Let us entertain the possibility that $P$ is such a polytope, i.e. it has a facet $F$ of the type $\Sigma_{3}$.

By \cref{twoout}, every vertex in $F$ belongs to only one edge not in $F$, and thus $P$ has 19=11+7+1 edges. According to \cref{thm:pentasm}, only the pentasm has 9 vertices and 19 edges, and it does not have $\Sigma_3$ as a facet according to \cref{rmk:Pentasm-Facets}.

(ii) Consider the possibility that $P$ is such a polytope, with $\Delta_{2,2}$ as a facet, say $F$. Again, by \cref{twoout},  every vertex in $F$ belongs to only one edge not in $F$, and so $P$ has 28=18+9+1 edges. But then $P$ has excess degree only one, which is impossible by ~\cref{thm:excess}.
\end{proof}

\begin{rmk}\label{rmk:Delta23-6Polytope}
A similar argument  proves that $\Delta_{2,3}$ is not a facet of any decomposable 6-polytope with 14 vertices; this will be useful to us in another context.
\end{rmk}

\begin{lemma}\label{twodeeplusone}
Let $P$ be a decomposable $d$-polytope with $2d+1$ vertices.
\begin{enumerate}[(i)]
\item Every facet of $P$ with fewer than $2d-2$ vertices is indecomposable. Every facet of $P$ with exactly $2d-2$ vertices is a prism. Any  facet of $P$  with $2d-1$ vertices is decomposable, and moreover every vertex therein belongs to only one edge outside the facet. No  facet of $P$ has $2d$ vertices.
\item If some decomposable facet of $P$ is  a capped prism, then $P$ is   a capped prism.
\item If some decomposable facet of $P$ is  a pentasm,  then $P$ is  a pentasm.
\item If every decomposable facet of $P$ is  a  prism, then  $P$ is either $\Sigma_3$, $\Delta_{2,2}$, or a capped prism $CP_{d,d}$.
\end{enumerate}
\end{lemma}
\begin{proof} (i) The indecomposability of $d$-polytopes with less than $2d$ vertices is asserted by \cite[Prop.~6]{Yos91}.

If a facet $F$ has $2d-2$ vertices then there are only three other vertices in $P$. If $d=3$, then $F$ is a quadrilateral, which is a 2-dimensional prism. If $d\ge4$, these three  vertices are not enough to form a facet. Then $F$ touches every facet, and must be decomposable by \cref{thm:McMullen-Decomp}, hence a prism by \cref{prop:2ddecomp}.

For facets  with $2d-1$ vertices, this is just \cref{twoout}.

The last assertion follows from the indecomposability of pyramids.

(ii) and (iii). Some facet $F$ is either a pentasm or a capped prism. We can write the vertex set of $F$ as $U\cup V$, where $U=\{u_i: 1\le i\le d-1\}$, $V=\{v_i: 0\le i\le d-1\}$ and the edges are as described in \cref{rmk:cappedprismstructure}; see also \cref{fig:polytopes}. Note that the geometric subgraphs determined by $U$ and $V$ are affinely independent cycles, and hence indecomposable.

Let $x$ and $y$ be the vertices outside $F$. We claim that one of them is adjacent only to vertices in $U$, while the other  is adjacent  only to vertices in $V$.

Suppose not. Then  one of them, say $x$, is adjacent to both $u_i$ and $v_j$ for some $i,j$. Since $d-1\ge3$, $x$ is adjacent to at least three vertices in $F$. Thus we may assume that $i\ne j$. If $F$ is a capped prism, or if $\{i,j\}\ne\{1,2\}$, then $x, v_j, v_i, u_i$ is a nonplanar 4-cycle sharing two vertices with $V$.  If $\{i,j\}=\{1,2\}$, then $x, v_j, v_3, u_3, u_i$ is an affinely independent 5-cycle sharing two vertices with $V$. Applying \cref{thm:McMullen-Decomp}, either $V\cup\{x,u_i\}$ or $V\cup\{x,u_i,u_3\}$ is  an indecomposable subgraph which contains all but $d-1$ or $d-2$ vertices of $P$, and hence touches every facet. But this implies indecomposability of $P$.

So we have $y$ (say) adjacent to every $v_j$, and $x$  adjacent to every $u_j$. This describes the graph of $P$ completely; it is the graph of either a pentasm or a capped prism.

We know already that a pentasm is determined uniquely by its graph, since it is the unique $d$-polytope with $2d+1$ and $d^2+d-1$ edges, except for $d=3$ where $\Sigma_3$ also share these properties; see \cref{thm:pentasm}.

In case $F$ is a capped prism, a little more care is needed to determine the face lattice of $P$, which we claim is also a capped prism; we may also assume that no other facet is a pentasm. Again, $x$ is adjacent to everything in $U$, and $y$ to everything in $V$, so it is reasonable to relabel  $x=u_d$ and $y=v_d$. Thus $P$ has the same graph as a capped prism, but we need to reconstruct the whole face lattice. Denote by $E_i$ the edge $[u_i,v_i]$ and by $k$ the minimum dimension of any face of the prism determined by $E_{1},\ldots,E_{d-1}$ whose affine hull contains $v_0$. Reindexing if necessary, we may suppose that $v_0$ lies in the affine hull of $\bigcup_{i=1}^kE_i$, and not to the affine hull of any smaller collection of these edges. For $j\le k$, let $F_j$ be the convex hull of $\bigcup_{i\ne j}E_i$, and let $S_j$ be the convex hull of $\{v_i: i\ne0,i\ne j\}$. For $j>k$, let $F_j$ be the convex hull of $\bigcup_{i\ne j}E_i\cup\{v_0\}$. It is not hard to see that each of these is a facet of $P$, as is $S$, the convex hull of $U\cup\{u_d\}$. In particular, $F=F_d$. We claim there are no other facets.

Since each $u_i$ is a simple vertex in $P$, it must belong to exactly $d$ facets of $P$. The list just given contains $d$ such facets for each $u_i$. So any other facet must be contained in $V\cup\{v_d\}$. All possible subsets are also accounted for by the list of facets $S$, $S_j$ for $j\le k$ and $F_j$ for $j\le d$. We have now described the vertex-facet incidences of a capped prism.

(iv) The case $d=3$ is essentially due to Smilansky \cite[\S6]{Smi87}; for further discussion, see \cite[p. 177]{PrzYos16}. Now suppose $d\ge4$.

Fix a decomposable  facet $F$, which by hypothesis is a prism. Now there are three vertices outside $F$; call them $a, b, c$. Following the previous notation, $F$ is the convex hull of two simplices $U$ (with vertices $u_1,\ldots,u_{d-1}$) and $V$ (with vertices $v_1,\ldots,v_{d-1}$). Note that if $i\ne j$, then no vertex outside $F$ can be adjacent to both $u_i$ and $v_j$. For if say $a$ were, then both $a,u_i,u_j,v_j$ and $a,u_i,v_i,v_j$ would be nonplanar 4-cycles with 3 vertices in common. Their union with $U$ and $V$ would then be an indecomposable graph containing all but two vertices of $P$, making $P$ indecomposable.

Consider the possibility that one of the external vertices, say $a$, is adjacent to both a vertex in $U$ and  a vertex in $V$. By the previous paragraph, there is a unique $i$ with $a$ adjacent to both $u_i$ and $v_i$.  Now the degree of $a$ is at most four, so $d=4$, and $F$ is 3-dimensional and contains three quadrilateral and two triangular ridges. Suppose $G$ is the other facet corresponding to one of the quadrilateral ridges. If $G$ were indecomposable, its union with the two triangular faces of $F$ would constitute an indecomposable graph containing at least seven of the nine vertices of $P$, contradicting the decomposability of $P$. It follows that each of these three other facets must be a prism. This means that each of them contains only two of the vertices outside $F$, each such vertex being adjacent to two vertices of the quadrilateral. This is only possible if $a,b,c$ can be renamed $w_1,w_2,w_3$ in such a way that they are adjacent to one another, and each $w_i$ is adjacent to both $u_i$ and $v_i$. Thus $P$ is simple and has the same graph as $\Delta_{2,2}$.

The remaining case is that each vertex outside $F$ is adjacent only to $U$ or $V$. Without loss of generality, $a$ is adjacent to every vertex in $U$, while $b$ and $c$ are adjacent only to vertices in $V$. Application of \cref{thm:McMullen-Decomp}(ii) ensures that $U\cup\{a\}$ and $V\cup\{b\}$ are indecomposable geometric graphs. Since $d\ge4$, we can choose an index $i$ for which $v_i$ is adjacent to both $b$ and $c$. We claim that $a$ cannot be adjacent to both $b$ and $c$. If it were, then $a,u_i,v_i,b$ and $a,u_i,v_i,c$ would both be  4-cycles in the graph of $P$; they cannot both be coplanar. So (at least) one of them, say the first, is an indecomposable geometric graph, which shares two vertices with the graph determined by $U\cup\{a\}$, and shares two vertices with the graph determined by $V\cup\{b\}$. The union of these three graphs is indecomposable, and contains every vertex except $c$. This would imply indecomposability of $P$, by \cref{thm:McMullen-Decomp}(iii).

Without loss of generality, $[a,b]$ and $[b,c]$ are edges of $P$, but $[a,c]$ is not. Having degree (at least) $d$, $c$ must then be adjacent to every vertex in $V$. If $b$ were adjacent to only $d-2$ vertices in $V$, we would have  only $d^2+d-1$ edges altogether, and then $P$ would be a pentasm by \cref{thm:pentasm}(ii). Otherwise, $b$ is also adjacent to every vertex in $V$, and we have the same graph as a capped prism. To verify the face lattice, let us first rename $a=:u_d$, $b=:v_d$ and $c=:v_0$. For fixed $i$, denote by $F_i$ the facet of $P$ containing every $u_j$ except $u_i$, by $R_i$ the convex hull of $\{u_j:1\le j\le d, j\ne i\}$, by $S_i$ the convex hull of $\{v_j:1\le j\le d, j\ne i\}$, and by $S_0$ the convex hull of $\{u_j:1\le j\le d\}$. Since each $u_i$ is simple, the $F_i$ are well defined, and $R_i$ is a ridge which has Shephard's property in $F_i$. Clearly $F_i$ contains $v_j$ whenever $j\ne i$, and $F_i$ is also decomposable, hence a prism. The second simplex $(d-2)$-face of $F_i$ is therefore $S_i$. For the ridge $S_i$, the other facet can only be the convex hull of   $S_i\cup\{v_0\}$; note that $v_0$ is also simple. We have now described the vertex-facet incidences of $CP_{d,d}$.\end{proof}

\begin{theorem}\label{thm:DavidDecom2d+1}
Let $P$ be a decomposable $d$-polytope with $2d+1$ vertices. Then $P$ is either a pentasm, a capped prism, $\Sigma_3$ or $\Delta_{2,2}$.\end {theorem}

\begin{proof} The  case $d=3$ is included in \cref{twodeeplusone}(iv). Henceforth assume $d\ge4$.

For $d=4$, suppose $P$ has nine vertices and is decomposable. Then according to \cref{thm:d-1DecomposableFacets} it has a decomposable  facet $F$, which clearly has at most seven vertices. By \cref{lowdim}(i), $F$ cannot be $\Sigma_3$, so it must be a prism, capped prism or pentasm. Every possibility is covered by the various cases in  \cref{twodeeplusone}, so $P$ is either $\Delta_{2,2}$, a capped prism or pentasm.

Likewise if $d=5$.  By \cref{lowdim}(ii), no decomposable facet $F$ can be $\Delta_{2,2}$, so the only options are prisms, capped prisms and pentasms. All these cases have been dealt with in \cref{twodeeplusone}.

Finally, we can  proceed with the induction. Suppose $d\ge6$, and that it has been established for every smaller dimension, and fix a $d$-dimensional decomposable polytope $P$ with $2d+1$ vertices.  Then $P$ has decomposable facet $F$, all with at most $2d-1$ vertices. By induction, each such facet must be  a prism, capped prism or pentasm, and \cref{twodeeplusone} then completes the proof.
\end{proof}

 The last proof incidentally proves that a conditionally decomposable $d$-polytope must have at least $2d+2$ vertices. It also leads to the following.

\begin{corollary}\label{cor:cappedgraph} A polytope, whose graph is that of a capped prism, is a capped prism.\end{corollary}

\begin{proof}

Let $P$ be such a $d$-polytope; then $P$ has $2d+1$ vertices and $d^2+d$ edges. We can label the vertices as $u_1,\ldots,u_d$, $v_0,\ldots, v_d$ as in \cref{rmk:cappedprismstructure}. We know that $u_1$ is simple, and that its neighbours are $u_2,\ldots,u_d,v_1$. It follows that one facet containing $u_1$ must contain $u_2,\ldots,u_d$ but not $v_1$. It is not hard to see that $u_1,\ldots,u_d$ then form a facet, with Shephard's property. Since pentasms, $\Sigma_3$ and $\Delta_{2,2}$ have strictly less than $d^2+d$ edges, $P$ must be a capped prism.
\end{proof}

\section{Characterisations of some 4-polytopes with minimum number of edges}

Recall that a 4-polytope with ten vertices must have at least 21 edges.
We give the complete characterisation of 4-polytopes with ten vertices and  21 or fewer edges; we will need this in \S7.  See \cref{fig:10v-21e} for drawings of these polytopes. They are  all decomposable.

We can also characterise the decomposable 4-polytopes with ten vertices and 22 edges. There are six such examples, as well as (at least) two indecomposable 4-polytopes with ten vertices and 22 edges; we do not need the details of this characterisation here though. More generally, we can characterise the polytopes with $2d+2$ vertices and minimal number of edges in all dimensions, but details will appear elsewhere \cite{PinUgoYos16}.

\begin{figure}[h]
\includegraphics[scale=1]{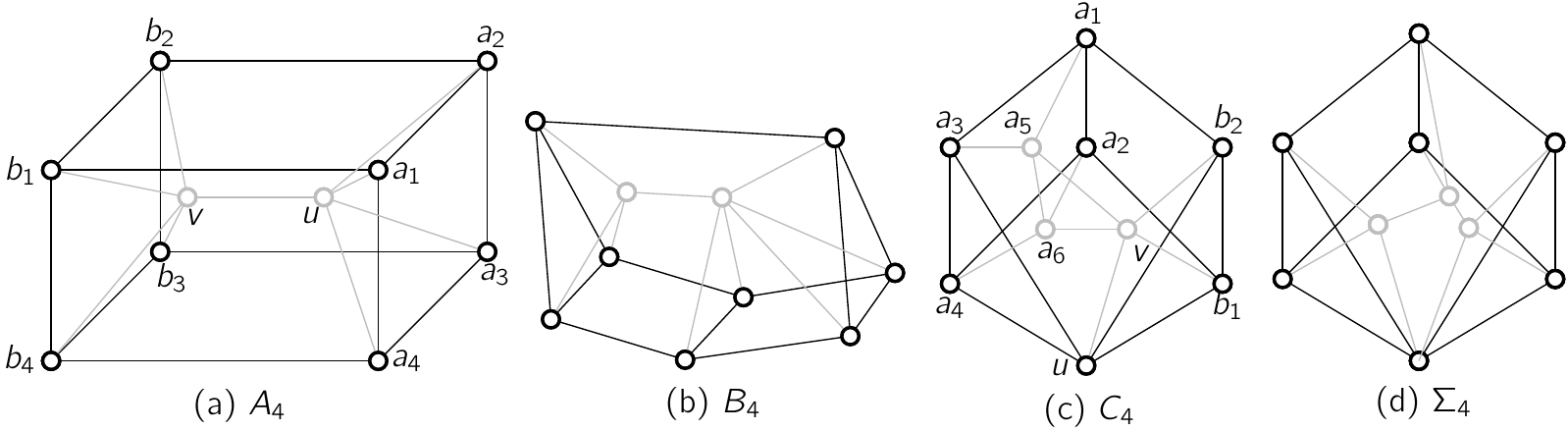}
\caption{The 4-polytopes with ten vertices and 21 or fewer edges.}
\label{fig:10v-21e}
\end{figure}

\begin{theorem}\label{thm:d=4} There are exactly four 4-polytopes with ten vertices and 21 or fewer edges: the polytopes $A_4$, $B_4$, $C_4$ and $\Sigma_4$, all of which have the same $f$-vector $(10,21,18,7)$. \end{theorem}

\begin{proof} Let $P$ be a 4-polytope with ten vertices and a minimum number of edges. According to \cite[Ch. 10]{Gru03} or \cref{lem:CountSimpleP}, $P$ is not simple, so we can assume it has exactly 21 edges. The excess degree of $P$ is exactly two, and consequently, any facet must have excess degree at most one thanks to ~\cref{lem:facet-not-all-excess}. This implies that any facet with an even number of vertices must be simple, and with an odd number of vertices must have excess one. In particular, any facet with eight or fewer vertices can only be a cube $\Delta_{1,1,1}$, a 5-wedge $J_3$, a pentasm, $\Sigma_3$, a prism $\Delta_{2,1}$, a quadrilateral pyramid $M_{2,1}$ or a simplex {\cref{lem:CountSimpleP}}. Note that no facet can have nine vertices, because it would have at least 14 edges, and $P$, being a pyramid thereover would have at least 23 edges, contrary to hypothesis. We now distinguish several cases.

{\bf Case 1.} Some facet $F$ is the cube.

Denote by $u$ and $v$ the two adjacent vertices outside $F$.

Consider a quadrilateral $R_1$ of $F$ and the quadrilateral $R_2$ of $F$ opposite to $R_1$. Let $F_1$ and $F_2$ be the other facets containing $R_1$ and $R_2$, respectively. If $F_1$ is a pyramid with apex $u$, $F_2$ is another pyramid with apex $v$. The face lattice of the resulting polytope is easily reconstructed from the graph; it is $A_4$ (cf.~\cref{rmk:Ad-Facets}). If no quadrilateral of $F$ is contained in a pyramid, then each quadrilateral must be contained in a simplicial 3-prism which contains both $u$ and $v$. A moment's reflection shows that this is impossible.

{\bf Case 2.} Some facet $F$ is a 5-wedge.

First note that every vertex in $F$ has exactly one neighbour outside $F$. Thus the other facet corresponding to each such ridge is a Shephard polytope. Denote by $u$ and $v$ the two adjacent vertices outside $F$.

Consider the other facet $F'$ containing one of the triangles. It cannot be a prism, as there are only two vertices outside $F$. Being a Shephard polytope, we see that $F'$  is tetrahedron with apex $u$. Likewise, the other facet containing the other triangle in $F$ is also a tetrahedron, with apex $v$.

As in  Case 1, the other facet corresponding to either quadrilateral face of $F$ must be a pyramid or a prism. It follows that the two remaining vertices on $F$ must be adjacent to the same vertex outside $F$, say $u$. Thus, one quadrilateral is contained in a pyramid and the other is contained in a 3-prism. The face lattice can be reconstructed with no difficulty from the graph obtained; the polytope is $B_4$, whose face lattice is detailed in \cref{rmk:Bd-Facets}.

{\bf Case 3.} Some facet $F$ is the 3-pentasm.

First note that the other facet containing the pentagon can only be a 5-wedge or another pentasm. If it  is a 5-wedge, we can refer to Case 2. The remaining option is that it  must be another pentasm. The two pentasms between them have 11 edges each, of which five are common to both. By \cref{thm:excess-d-2-full-story}, the nonsimple vertex in each pentasm must be adjacent to the nonsimple vertex in the other pentasm. Thus the union of the two pentasms contains 18 edges altogether. The remaining vertex outside $F$ has degree at least four, giving at least 22 edges altogether. This cannot be.

{\bf Case 4.} Some facet $F$ is $\Sigma_3$.

We may assume that no facet is a cube, 5-wedge or pentasm. Denote by $u$ the vertex of excess degree one in $F$.

From \cref{thm:excess-d-2-full-story} it ensues that $P$ is decomposable. Let us fix a triangular ridge $T$ in $F$ and consider the other facet $F'$ containing it. If $F'$ were a quadrilateral pyramid, then every vertex outside $F$ would be simple in $P$, and thus, every facet of $P$ would touch $F'$, implying by \cref{thm:McMullen-Decomp} that $P$ is indecomposable, a contradiction.

Thus the two other facets corresponding to the triangular ridges in $F$ can only be simplices or prisms.  If both are prisms, they contain a common triangle (the three vertices outside $F$) and a common edge (the one from $u$ running outside out of $F$), meaning they must lie in the same 3-dimensional affine space. This is also impossible.

Suppose both are simplices. Then one of the vertices outside $F$ will be adjacent to all three vertices in one the two triangles, another  will be adjacent to all three vertices in the other triangle, and the third  will be adjacent to both of the other vertices in $F$. It is routine to verify that $P$ has the same face lattice as $\Sigma_4$.

In the remaining case, there must be exactly seven edges leaving $F$ and the three vertices outside $F$ are pairwise adjacent. We can denote one of them by $v$, which sends three edges to $F$; while the remaining two vertices each send two edges to $F$. Note that the vertex $v$ must be adjacent to the vertex $u$ by \cref{lem:simpleVertexOutside}. The only possible pattern of connections between $F$ and the external vertices gives us the graph and face lattice of $C_4$; see \cref{rmk:Cd-Facets} and \cref{fig:10v-21e}.

{\bf Case 5.} Every facet is a prism, a quadrilateral pyramid or a simplex.

By \cref{thm:excess-d-2-full-story} no every facet can be simple. So it suffices to show that no facet can be a pyramid in this case. Suppose some facet $F$ is a pyramid based on a quadrilateral ridge $Q$. Denote by $u$ the apex of this pyramid and let $G$ be the other facet containing $Q$. Then $G\setminus Q$ contains at most two vertices. If $u$ has degree six in $P$, then every other vertex in $P$ is simple. If $u$ has degree five in $P$, then the unique vertex adjacent to $u$ outside $F$ must be nonsimple (recall \cref{lem:simpleVertexOutside}). In either case, every vertex in $Q$ is simple in $P$. Thus all the neighbours of vertices of $Q$ are in $F\cup G$. If we remove the vertices in $(G\setminus Q)\cup\{u\}$ from the graph of $P$, the resulting graph will be disconnected, contrary to Balinski's Theorem.
\end{proof}

\section{$(f_0,f_1)$-projections of 5-polytopes}

We characterise all pairs $(f_0,f_1)$ for which there exists a 5-polytope with $f_0$ vertices and $f_1$ edges. As in \cite{PinUgoYos15}, let us define $E(v,d)=\{e:$ there is a $d$-polytope with $v$ vertices and $e$ edges$\}.$ In this notation, we determine $E(v,5)$ precisely, for all values of $v$.
In particular, we show that $\min E(f_0,5) = \h(5f_0 + 3)$ if $f_0$ is odd, and $\min E(f_0,5) =  \frac{5}{2}f_0$ if
$f_0$ is even and not eight. It is well known that $\min E(8,5) =  22$.

\begin{lemma}
Besides the simplex, there is a simple 5-polytope with $f_0$ vertices if, and only if, $f_0$ is even and $f_0\ge10$.
\end{lemma}

\begin{proof}
Simplicity requires $2f_1=5f_0$, so $f_0$ must be even. The simplex $\Delta_{5}$, the prism $\Delta_{1,4}$ and our other friend $\Delta_{2,3}$ are all simple, with 6, 10 and 12 vertices respectively. Truncating one vertex of any simple 5-polytope gives another simple 5-polytope with four more vertices and ten more edges. Thus we obtain all even numbers from ten onwards. There is no example with eight vertices because $8<2d$.
\end{proof}

\begin{theorem}
There is a nonsimple
5-polytope with $f_0$ vertices and $f_1$ edges, if, and only if, $$\h(5f_0 + 3) \le f_1 \le {f_0\choose2}$$
and $(f_0,f_1)\ne (9, 25),(13, 35)$.

\end{theorem}

\begin{proof} Clearly we must have $\frac{5}{2}f_0\leq f_1 \leq {f_0 \choose 2}$. Any pair with $2f_1-5f_0=1$ or 2 is impossible by the excess theorem (\cref{thm:excess}). The case $(9,25)$ is also impossible by \cite[Theorem 19]{PinUgoYos15}; the proof of that result becomes somewhat shorter if we set $d=5$.

All other pairs except $(13,35)$ are possible. How to construct them? The first step is to look at all 4-polytopes and construct pyramids over them. If a $d$-polytope has $f_0$ vertices and $f_1$ edges, then a pyramid thereover is a $(d+1)$-polytope with $f_0+1$ vertices and $f_0+f_1$ edges. Gr\"unbaum \cite[Theorem 10.4.2]{Gru03} showed that
$E(6,4)=[13,15]$,
$E(7,4)=[15,21]$,
$E(8,4)=\{16\}\cup[18,28]$,
$E(9,4)=[18,36]$,
$E(10,4)=[21,45]$, and
$E(f_0,4)=[2f_0,{f_0\choose2}]$ for all $f_0\ge11$.

For $f_0\le 11$,  building pyramids on 4-polytopes with $f_0-1$ vertices shows that
$E(7,5)=[19,21]$,
$E(8,5)\supseteq[22,28]$,
$E(9,5)\supseteq\{24\}\cup[26,36]$,
$E(10,5)\supseteq[27,45]$, and
$E(11,5)\supseteq[31,55]$.

Thus we have all the alleged examples for $f_0\le11$ except $(11,29)$ and $(11,30)$. But these are exemplified by the   pentasm and the capped prisms respectively. Note that these three examples all come from slicing a simple vertex off something else:  a triplex $M_{2,3}$ or a bipyramid over a 4-simplex, respectively.

For $f_0\geq 12$, pyramids give all examples with $f_1\geq 3f_0-3$. More precisely, for each such $f_0$, we know that $E(f_0-1,4)=[2f_0-2,{f_0-1\choose2}]$ and hence $E(f_0,5)\supset[3f_0-3,{f_0\choose2}]$. The cases $\h(5f_0 + 3)\leq f_1<3f_0-3$ require a little explanation.

Note that if $f_1<3f_0$, then a 5-polytope with $f_0$ vertices and $f_1$ edges must have at least one simple vertex.

So suppose $f_0\ge12$ and $\h(5f_0 + 3)\leq f_1<3f_0-3$.  Let $k$ be the smallest integer $\ge\h(3f_0-3-f_1)$. Then $f_1-3f_0+2k$ is either $-3$ or $-2$. Now put $f_1'=f_1-10k$ and $f_0'=f_0-4k$. Clearly $3f_0'-3\le f_1'<3f_0'$, so there is a polytope with $f_1'$ edges and $f_0'$ vertices, at least one of which is simple, unless $f_1'=25$  and $f_0'=9$. Truncating a simple vertex $k$ times then gives a polytope with $f_1$ edges and $f_0$ vertices.

The case $f_0=17,f_1=45$ comes from slicing a simple edge off a capped 5-prism; note that truncating a simple edge increases $f_0$ by six, and $f_1$ by 16. All remaining cases, i.e. $f_0=17+4k,f_1=45+10k$ for suitable $k$, then come from repeated truncation of simple vertices.

It only remains to prove the unfeasibility of $(f_0,f_1)=(13,35)$. This has a lengthy proof, to be given  separately.
\end{proof}

We formally state this impossibility result here, but postpone its proof  to  \cref{sec:13n-35e}.

\begin{theorem}
\label{thm:v13-e35}
There is no 5-polytope with $13$ vertices and 35 edges.
\end{theorem}

%---REFERENCES---
%\bibliographystyle{amsplain}
%\bibliography{GPinedaVillavicencioBibDesk}

%%\providecommand{\bysame}{\leavevmode\hbox to3em{\hrulefill}\thinspace}
%%\providecommand{\MR}{\relax\ifhmode\unskip\space\fi MR }
%%\providecommand{\MRhref}[2]{%  \href{http://www.ams.org/mathscinet-getitem?mr=#1}{#2}}
%%\providecommand{\href}[2]{#2}

\appendix
\section{Unfeasibility of $(d,f_0,f_1)=(5,13,35)$}
\label{sec:13n-35e}

This appendix completes the proof of \cref{thm:v13-e35}. We proceed by examining  the numbers of vertices in the facets of such a hypothetical polytope. The proof of Kusunoki and Murai \cite {KusMur17}, is quite different, considering  instead the possible degrees of the vertices of the polytope. We first require the main result from \cite{PinUgoYos15}, which considered polytopes with a fixed number of vertices, not exceeding $2d$, and characterised those with the minimal number of edges. It is convenient to rephrase this result in terms of the excess degree.

\begin{theorem}[{\cite[Thm.~6]{PinUgoYos15}}]\label{thm:triplexes} Let $P$ be a $d$-polytope with $d+k$ vertices with $1\le k\le d$. Then  $$\xi(P)\ge(k-1)(d-k),$$ ( equivalently $f_1(P)\ge\h d(d+k)+\h(k-1)(d-k)$), and equality is obtained if, and only if, $P$ is a $(k,d-k)$-triplex.\end{theorem}

We   now proceed via a number of lemmas, which restrict the options for facets for hypothetical 5-polytopes $P$ with $13$ vertices and 35 edges.

\begin{lemma}
\label{lem:4-polytope-11v-23e}
Any 4-polytope $P$ with 11 vertices and 23 edges must contain a facet with 8 vertices. Moreover, the facet contains only simple vertices, and  the other three vertices form a triangular face.
\end{lemma}
\begin{proof}
Suppose otherwise; then every 3-face in $P$ has at most seven vertices.
Since the excess of $P$ is two, \cref{thm:excess-d-2-full-story} tells us that it is decomposable and has a facet with excess one. Moreover, \cref{lem:facet-not-all-excess} tells us that no facet has excess two, so every facet of $P$ with an odd number of vertices has excess one.

Suppose there is a 3-face $F$ with seven vertices. Then $F$ must have 11 edges, and hence there at least seven edges leaving $F$, and at most five edges joining the four vertices outside $F$. In particular the four vertices outside $F$ do not constitute a facet, and $F$ thus intersects every facet. Since $F$ is not simple, \cref{lem:Non-DisjointFacets}(ii) implies that $F$ has Kirkman's property, i.e. it intersects every other facet in a ridge. Since $F$ has six ridges, we conclude that $P$ has seven facets.

In particular, $P$ has $f$-vector $(11,23,19,7)$. It is known that there are only three combinatorial types of polytope with $f$-vector $(7,19,23,11)$; their Gale diagrams appear in \cite[Fig. 5]{Gru70}, and we have reproduced them here in \cref{fig:11-23-19-7}. Our $P$ must be dual to the polytopes corresponding to these diagrams. Studying them then reveals the face lattice of $P$; we omit the routine calculations. For the record, we note that the facets of these three examples are respectively

\begin{figure}
\includegraphics{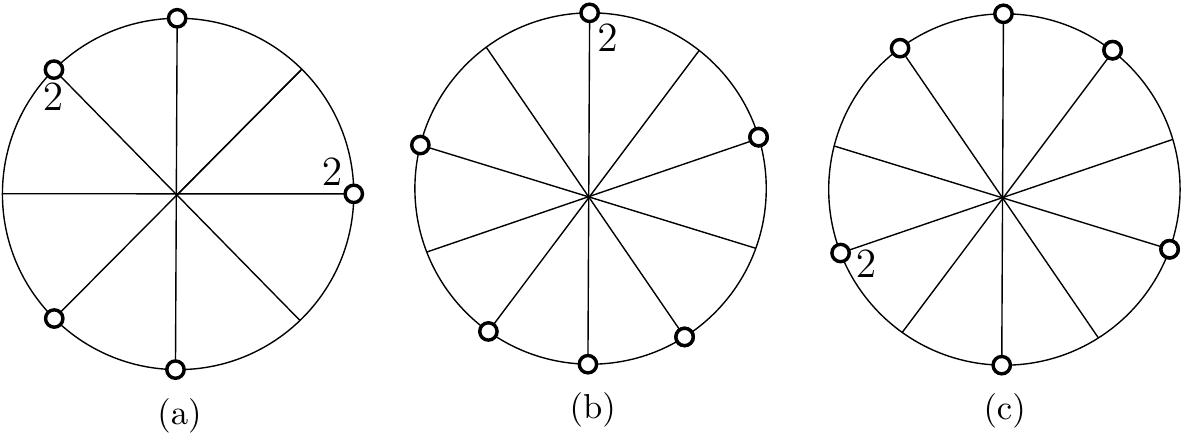}
\caption{Gale diagrams of the 4-polytopes with $f$-vector $(7,19,23,11)$.}\label{fig:11-23-19-7}
\end{figure}

\noindent (a) a cube, two copies of $\Sigma_3$, three prisms and a quadrilateral pyramid;

\noindent (b) a 5-wedge, two pentasms and four prisms;

\noindent (c) a 5-wedge, two pentasms, two prisms,  $\Sigma_3$ and a simplex.

A diagram  $P$ of case (a) appears in \cite[10.4.5]{Gru70}; it is easy to see that either copy of $\Sigma_3$ must intersect the cube in a quadrilateral face, containing only simple vertices; thus the other three vertices must be a triangular face. The same reasoning applies in case (c). In case (b), the pentasms must be other facets corresponding to the pentagonal ridges, and so the other facet corresponding to either triangle in the 5-wedge must a prism, and its opposite end is a triangular face disjoint from the 5-wedge.

Finally, suppose there is a facet $F$ with five vertices, but no facet with seven vertices; we will see that this is impossible. Such an $F$ has excess one and is a quadrilateral pyramid with apex $u$.
Suppose there is a nonsimple vertex $v$ in $P\setminus F$. Then, the degrees of $u$ and $v$ in $P$ are five, $u$ and $v$ are adjacent by \cref{lem:simpleVertexOutside}, and there are exactly five edges between $F$ and $P\setminus F$. Consider the other facet $G$ containing the quadrilateral $Q$ of $F$. Then either $G$ is a quadrilateral pyramid with apex $v$ or a simplicial 3-prism. In the former case, the graph of $F\cup G$ wil be disconnected from the rest of $P$. In the latter case, removing  $u$ and the two vertices $G\setminus F$  would disconnect $Q$ from the other vertices of $P$, contradicting Balinski's Theorem.

If instead every vertex in $P\setminus F$ is simple, the degree of $u$ in $P$ will be six   and there will be exactly six edges between $F$ and $P\setminus F$. Again, considering the other facet $G$ containing the quadrilateral $Q$ of $F$, it follows that $G$ is a simplicial 3-prism. But then, removing $u$ and the two vertices $G\setminus F$ would disconnect $Q$ from the rest of $P$, again contradicting Balinski's Theorem.
\end{proof}

\begin{lemma}\label{lem:10v} If a 5-polytope $P$ has $13$ vertices and 35 edges, and $F$ is a facet with ten vertices, then $F$ has 21 edges.
\end{lemma}
\begin{proof}
Note that $F$ has Kirkman's property; otherwise the existence of a facet of  intersecting $F$ other than at a ridge would force the existence of at least thirteen edges between $F$ and  the three vertices outside, say $x,y,z$, giving a minimum  of 37  edges of in $P$. This property implies that the number of facets other than $F$ in $P$ coincides with the number of 3-faces of $F$, and that every vertex simple in $F$ is simple in $P$.

The excess of $P$ is five, so the excess of $F$ must be either zero, two or four. We know it cannot be simple by \cref{lem:CountSimpleP}(iv). Suppose the excess is four; then $F$ has 22 edges, and there are 13 edges  outside $F$. Let $x$, $y$ and $z$ denote the vertices in $P\setminus F$. We can assume that $x$ and $y$, and $x$ and $z$ are adjacent.

First suppose that $y$ and $z$ are not adjacent. Then there is a unique vertex $u$ in $F$ sending two edges outside $F$.  The vertices $x$, $y$ and $z$ are all simple. If there were another nonsimple vertex in $F$, then, by \cref{lem:simpleVertexOutside}, the vertex would have two neighbours outside $F$. Consequently, $u$ is the  unique nonsimple vertex of $F$ (and of $P$), with degree ten in $P$. But then, \cref{lem:vertex-all-neighbours-simple}  would imply that the vertex figure of $P$ at $u$  is a simple 4-polytope with ten vertices, which we know to be impossible.

Now suppose that $y$ and $z$ are  adjacent, with $y$ having degree six in $P$. In this case, there are ten edges leaving $F$. Since the vertices $x$ and $z$ are simple, there is a facet $G$ in $P$ containing $x$ and $z$ but not $y$. By \cref{lem:simpleVertexOutside} all vertices in $R:=G\cap F$ are simple in $P$.  If  $R$ has six vertices, it must  be a simplicial 3-prism. Looking at the structure of $F$, we see that nine edges of $F$ come from $R$ and that there are six edges between $R$ and the four vertices in $F\setminus R$. But this counting leaves $7>{4\choose 2}$ edges to be distributed among the four vertices in $F\setminus R$, another contradiction.
Being simple, $R$ must then have eight vertices. But then, there would only be eight edges between $R$ and the two vertices in $F\setminus R$, meaning that $F$ would only have 21 edges.

The remaining case is that $F$ has excess degree two, and thus 21 edges.
\end{proof}

\begin{lemma}\label{lem:9v} Suppose that $P$ is a  5-polytope with $13$ vertices and 35 edges, and that $F$ is a facet with 9 vertices. Then $F$ is either a pentasm or $\Delta_{2,2}$.
\end{lemma}
\begin{proof}
Since $P$ has excess five, it is clear that the 4-dimensional facet $F$ either has excess zero (18 edges), excess two (19 edges) or excess four (20 edges). Denote by $x_1$, $x_2$, $x_3$ and $x_4$ the vertices outside $F$.  We need to eliminate the possibility that the facet $F$ has 20 edges; suppose it has.

If every vertex in $P\setminus F$ were simple in $P$, then by \cref{lem:simpleVertexOutside} every nonsimple vertex of $P$ which is contained in $F$ would send two edges outside $F$. In this case, there would be a unique nonsimple vertex in $F$ (and in $P$), say $u$, with degree ten in $P$. But then, \cref{lem:vertex-all-neighbours-simple}  gives that the vertex figure of $P$ at $u$ would be a simple 4-polytope with ten vertices, which is ruled out by \cref{lem:CountSimpleP}.

Consequently, there is a unique nonsimple vertex  in $P\setminus F$, say $x_1$, which has degree six in $P$ and is adjacent to to every nonsimple vertex lying in $F$. Furthermore, there are exactly nine edges between $P\setminus F$ and $F$, and exactly six edges shared among the vertices $x_1$, $x_2$, $x_3$ and $x_4$.  We count the number of facets involving the vertices $x_1$, $x_2$, $x_3$ and $x_4$. There are exactly five facets containing $x_3$, since $x_3$ is simple. Since $x_4$ is simple, there is exactly one facet containing $x_4$ but not $x_3$; this facet also contains $x_1$ and $x_2$. In this way, we have counted all the facets containing the vertices $x_2$, $x_3$ and $x_4$. Out of of the aforementioned six facets, the nonsimple vertex {$x_1$ cannot be contained in all of them, since there must exist a facet containing $x_3$ but not $x_1$ (recall \cref{rmk:two-vertices}). But then the nonsimple vertex $x_1$} is only contained in five facets, a contradiction.

By \cref{thm:pentasm},  $F$  is either a pentasm or $\Delta_{2,2}$.
\end{proof}

\begin{lemma}\label{lem:8v} Suppose that $P$ is a  5-polytope with $13$ vertices and 35 edges, and that $F$ is a facet with 8 vertices. Then $F$ is a prism.
\end{lemma}
\begin{proof}
Clearly the excess of $F$ can only be zero, two  or four. From \cite[\S10.4]{Gru03},  it follows that $F$ cannot have 17 edges, i.e. its excess cannot equal two (see \cite[Thm. 19]{PinUgoYos15} for a higher dimensional version of this).

So suppose that $F$ has excess four, i.e. there are 18 edges in $F$. Then   $F$ is indecomposable by \cref{prop:2ddecomp}, and the five vertices $\{x_1,x_2,x_3,x_4,x_5\}$ outside $F$ cannot form a facet, as $P$ would then have at least 36 edges. Then  every facet of $P$ touches $F$, and so the whole polytope is indecomposable by \cref{thm:McMullen-Decomp}(iii). In particular, $F$ cannot have Shephard's property by \cref{thm:shp}, thus there are nine edges running out of $F$, edges edges joining the $x_i$, each $x_i$ is simple. Since every nonsimple vertex of $P$ which is contained in $F$ must send two edges outside $F$, there would then be a unique nonsimple vertex in $P$ (actually in $F$),  which can only have degree ten in $P$. But then, \cref{lem:vertex-all-neighbours-simple}  gives that the corresponding vertex figure of $P$ would be a simple 4-polytope with ten vertices, which is ruled out by \cref{lem:CountSimpleP}.

The only remaining possibility is that $F$ is simple, i.e. a simplicial 4-prism.
\end{proof}

\begin{theorem}
\label{thm:v13-e35}
There is no 5-polytope with $13$ vertices and 35 edges.
\end{theorem}
\begin{proof} Let $P$ denote a $5$-polytope with 13 vertices and 35 edges. Throughout, $F$ will denote a fixed facet with the maximal number of vertices. We investigate various cases, based on this number.

{\bf Case 1.} The facet $F$ has 11 or 12 vertices.

This case of 12 vertices is clearly impossible. such a facet $F$ would have at least $(12\times 4)/2$ edges, causing $P$ to have at least 36 edges.

Let $F$ denote a facet of $P$ with 11 vertices. The facet cannot have more than 23 edges, since there must be at least twelve edges outside $F$. Thus the facet $F$ has 22 or 23 edges. In either case, $F$ has Kirkman's property; indeed, if another facet $G$ intersected $F$ at a subridge or face of smaller dimension,  then at least fourteen edges would leave $F$.
If  $F$ has 22 edges, it is simple, and by \cref{lem:wksp} every vertex in $F$ is also  simple in $P$. But then $P$ would only have 22+11+1 edges.
Thus $F$ has 23 edges, there must be exactly eleven edges leaving $F$, so $F$ has Shephard's property and $P$ is decomposable.

Now \cref{lem:4-polytope-11v-23e} ensures the existence of a ridge $R$ in $F$ with eight vertices, in which each vertex is simple. Let $G$ denote the other facet containing $R$. It cannot be a pyramid, because its apex would have excess degree four, while $F$ has excess degree two. Thus $G$ has ten vertices.

By \cref{lem:4-polytope-11v-23e}, the three vertices in $F\setminus R=P\setminus G$ form a triangular face $T$. Denote by $u,v$ the two vertices in $G\setminus R$. One of them, say $u$, must be adjacent to two vertices in $T$, which means that the geometric graph with vertices $T\cup\{u\}$ is indecomposable. Now $v$ is adjacent to the third vertex in $T$, and also to $v$, whence the graph with vertices $T\cup\{u,v\}$ is indecomposable. The complement of this graph contains only vertices in $R$, and so the graph touches every facet. This implies that $P$ is indecomposable, which is a contradiction.

{\bf Case 2.} The facet $F$ has 10  vertices.

We know from \cref{lem:10v} that $F$ has 21 edges. It must also have Kirkman's property, otherwise there would be at least 13 edges leaving $F$. From
\S6, we also know that $F$ must be $A_4,B_4,C_4$ or $\Sigma_4$. In each case $F$ has seven 3-faces, so $P$ must have eight facets.

We first consider the possibility that $F$ is either $A_4$ or $B_4$. Then there is a ridge $R$ in $F$ containing eight vertices, all of them simple. We consider the other facet $G$ containing this ridge. It cannot be a pyramid, because its apex would have excess degree four, while $F$ has excess degree two. Nor can it have eleven vertices, by hypothesis. Thus $G$ has ten vertices, and there is a unique vertex outside $F\cup G$. This vertex cannot be connected to any vertex in $R$, because they are all simple and all have degree five in the graph $F\cup G$. Thus removal of the four vertices in $(F\cup G)\setminus R$ disconnects the graph of $P$, contradicting Balinski's Theorem.

Now suppose $F$ is either $C_4$ or $\Sigma_4$. Then $F$ contains a ridge $R_1$ which is $\Sigma_3$; we denote the other facet containing $R_1$ by $G$, and consider several possible cases. First note that $G$ cannot contain eight vertices, because then it would be a pyramid, violating \cref{lem:8v}. By \cref{lem:9v}, it cannot contain nine vertices, because neither $\Delta_{2,2}$ nor a 4-pentasm can contain $\Sigma_3$. By assumption, $G$ does not contain eleven or more vertices.

So $G$ also has ten vertices, and must be either $C_4$ or $\Sigma_4$. We show that it cannot be $\Sigma_4$. Denote by $x,y,z$ the three vertices of $G$ not contained in $R_1$; if $G$ is $\Sigma_4$, there will only be two edges between them. Suppose that $y$ and $z$ are not adjacent. Renaming the vertices if necessary, there are just two possibilities: either the vertex $x$ or the vertex $y$ has degree six in $P$. In the first case, there will be at least six facets containing $x$, since $x$ is nonsimple. And there must be two further facets in $P$, one containing $y$ but not $x$ or $z$, and another one containing $z$ but not $x$ or $y$. Including $F$, this means that $P$ has nine facets altogether. Analogously, in the second case, there are at least six facets containing $y$, since $y$ is nonsimple. And apart from $F$, there must be two further facets in $P$, one containing $x$ and $z$ but not $y$, and another one containing $z$ but not $x$ or $y$. Again, we have the contradictory conclusion that $P$ has nine facets.

So every facet with 10 vertices is $C_4$, and the three vertices outside are connected by three edges.

For any copy of $C_4$, its vertices can partitioned  into three subsets $Q, T_1, T_2$, uniquely up to renaming $T_1, T_2$, such that $Q$ constitutes a quadrilateral face, each  $T_i$ is a triangular face, the convex hull of $T_1\cup T_2$ is a 3-prism, and for each $i$ and the convex hull of $Q\cup T_i$ is a 3-face of $F$ equivalent to $\Sigma_3$. In terms of \cref{fig:10v-21e}(c), we can take $Q=\{a_1,a_2,b_1b_2\}$, $T_1=\{u,a_3,a_4\}$ and $T_1=\{v,a_5,a_6\}$.

Now fix   $Q, T_1, T_2$ as the corresponding sets in $F$, $R_1$ being the convex hull of $Q\cup T_1$.
Now $G$ must also be a copy of $C_4$, with the corresponding partition being  $Q, T_1, T_3$, where $T_3$ comprises the three vertices outside $F$. Considering other facets containing ridges, we see that convex hull of  $Q, T_2$ and $T_3$ will also be facet of the form $C_4$. Then any edge of $P$ will be contained in one of the ridges determined by $Q\cup T_i$ for some $i$, or by $T_i\cup T_j$ for distinct $i,j$. Adding these up, we find that $P$ has only 34 edges, contrary to hypothesis.

{\bf Case 3.} The facet $F$ has nine vertices. By \cref{lem:9v}, $F$ is either a pentasm or $\Delta_{2,2}$.

{\bf Subcase 3.1} The facet $F$  is $\Delta_{2,2}$.

Then $F$ has six 3-faces, all of them being simplicial 3-prisms. See ~\cref{fig:polytopes}(c), where the vertices of $F$ are labelled  so that two of them are adjacent if, and only if, they share a number or a letter.
Then the facets of $F$ (which are ridges in $P$) are the convex hulls of the six subsets which either omit one number or omit one letter. Accordingly we  denote them as $R_{12}$, $R_{13}$, $R_{23}$, $R_{ab}$, $R_{ac}$ and $R_{bc}$.
Thus the 3-faces in $\Delta_{2,2}$ are partitioned into two groups of three; any two 3-faces within the same group intersect in a triangle, and any two 3-faces from different groups intersect in a quadrilateral.

There are 17 edges outside $F$; denote the vertices outside by $x_1$, $x_2$, $x_3$ and $x_4$. Any subset of three of these vertices will belong to at least 12 edges; hence any one of them can be adjacent to at most five vertices in $F$. It follows that, for any 3-face $R$ in $F$, the corresponding other facet is never a pyramid over $R$. This in turns implies that each $x_i$ is adjacent to at least two of the others (if there were say only one edge from $x_1$ to the other $x_i$, the only facet containing $x_1$ but not this edge would be a pyramid with apex $x_1$ and base in $F$).

For any ridge $R$ in $F$, the corresponding other facet must have either eight or nine vertices, and so must be either a 4-prism (by \cref{lem:8v}) or a copy of $\Delta_{2,2}$. So if $T$ is any triangular face in $F$, any one of the external vertices $x_i$ must be adjacent to either none, one (in the case of $\Delta_{2,2}$) or three (in the case of a 4-prism) of the vertices in $T$; two is not possible.

Suppose that the other facet in $P$ containing $R_{12}$ is a prism, and that the other facet containing $R_{13}$ is another $\Delta_{2,2}$. We may suppose that $x_i$ is adjacent to $a_i$, $b_i$ and $c_i$ for $i=1,2$. Then $x_2$ clearly does not belong to the other facet containing $R_{13}$, so the other three vertices in this facet must $x_1,x_3,x_4$. But  $x_1$  is adjacent to all vertices in one of the triangles in $R_{13}$, which makes it impossible for the other facet containing $R_{13}$ to be a copy of $\Delta_{2,2}$.

It follows that the other facets corresponding to  $R_{12}$, $R_{13}$ and $R_{23}$ are either all prisms, or all $\Delta_{2,2}$. Likewise for $R_{ab}$, $R_{ac}$ and $R_{bc}$. This can only be achieved (after relabelling the $x_i$ if necessary) if it holds that for $i=1,2,3$, $x_i$ is adjacent to $a_i$, $b_i$ and $c_i$, but not to  $a_j$, $b_j$ or $c_j$ for $j\ne i$, and that  $x_1$, $x_2$, $x_3$ are mutually adjacent. Three of these other facets are 4-prisms, and three of them are copies of  $\Delta_{2,2}$. Furthermore, none of them contain $x_4$: this will lead to an absurdity.

Suppose $G_1$ is any facet containing $x_4$. Clearly  $G_1$  intersects $F$ at a non-empty face of dimension $<3$. Every vertex in $F$ is adjacent to precisely one of $x_1$, $x_2$, $x_3$; this rules out the possibility that $F\cap G_1$ is a single vertex or an edge. Were   $F\cap G_1$ a quadrilateral, each of its vertices would be adjacent to two vertices outside $F$, and there would be at least 13 edges running out from $F$. But then $x_4$ could only be adjacent to one of the other $x_i$, contrary to our previous conclusion. Thus  $F\cap G_1=T_1$ must be a triangle.

Now let $G_2$ be another facet containing $x_4$, but not containing $T_1$. Then  $F\cap G_2=T_2$ is another triangle, each of whose vertices is also adjacent to two vertices outside $F$. Thus there would be at least 10 edges running out from $F$ from $T_1\cup T_2$, and another 4 from the remaining vertices. But then $x_4$ could not be adjacent to any of the other $x_i$.

{\bf Subcase 3.2} The facet $F$ is a $4$-pentasm.

Then $F$ has a 3-face $R$ of $F$ which is a 3-pentasm (see \cref{rmk:Pentasm-Facets}), in particular, $R$ has seven vertices and eleven edges. Denote by $G$ the other facet containing $R$. We know that $G$ cannot have ten or more vertices. If it has 9 vertices, and it cannot be $\Delta_{2,2}$ because it contains a pentagonal face, and so has at least 19 edges. But then $F\cup G$ contains at least 27 edges, while the two vertices outside $F\cup G$ must contribute at least nine more edges, which is too many. Finally, consider the possibility that $G$ contains eight vertices. Then $G$ is a pyramid over $R$, $F\cup G$ contains  26 edges, and the three vertices outside $F\cup G$ must contribute at least 12 more edges, which is again too many.

{\bf Case 4.} The facet $F$ has eight vertices.

Denote the vertices outside $F$ by $x_1,x_2,x_3,x_4,x_5$. By \cref{lem:8v}, $F$ itself is a 4-prism.

Observe that every two simplicial 3-prisms in $F$ intersect in a quadrilateral 2-face. Consider one of the simplicial 3-prisms $R_1$ in $F$. If the other facet $F_1$ containing $R_1$ were a pyramid, then $F_1\cup F$ would contain 22 edges and there would be four vertices outside $F_1\cup F$ incident to 13 edges; this is impossible since four vertices are incident to at least 14 edges. So $F_1$ (and every facet containing of  one of the simplicial 3-prisms in $F$) must also be a simplicial 4-prism.  Let $x_1$ and $x_2$ denote the vertices in $F_1\setminus R_1$. Consequently, there are 12 edges outside $F_1\cup F$ incident to the remaining three vertices $x_3$, $x_4$ and $x_5$. So every pair in $\{x_3,x_4,x_5\}$ is adjacent, each vertex in $\{x_3,x_4,x_5\}$ is simple, and at least one vertex in $\{x_3,x_4,x_5\}$ is adjacent to a vertex in $\{x_1,x_2\}$. This in turn implies that every vertex  $\{x_1, x_2,x_3,x_4,x_5\}$ sends at most three edges into $F$. The facet $F$ has four different simplicial 3-prisms, say $R_1$, $R_2$, $R_3$ and $R_4$, and correspondingly, four different  simplicial 4-prisms $F_1$, $F_2$, $F_3$ and $F_4$ containing these 3-faces. If a vertex in  $\{x_1, x_2,x_3,x_4,x_5\}$ were contained in more than one of the facets $F_1$, $F_2$, $F_3$ and $F_4$, it would send at least four edges into $F$, a contradiction. Consequently, the pairs of vertices in  $F_i\setminus R_i$, for $i=1,2,3,4$, must be pairwise disjoint, which is clearly a contradiction.

{\bf Case 5.} The facet $F$ has seven vertices.

The minimum number of edges of $F$ is 15 by \cref{thm:triplexes}, in which case $F=M_{3,1}$, that is, a pyramid over a simplicial 3-prism $R$. Suppose $F$ has 15 edges and consider the other facet $F_1$ containing $R$ . Then   $F_1=M_{3,1}$, and $F\cup F_1$ contains 21 edges. Hence there are 14 edges outside  $F\cup F_1$ which are incident to five vertices of degree at least 5, which is impossible.  Thus $F$ has at least 16 edges.

Suppose $F$ has exactly 16 edges. If $F$ has a pentagonal face, it is obviously a 2-fold pyramid. Otherwise, denoting by $q$ the number of quadrilateral faces in $F$, Kalai's generalisation of the Lower Bound Theorem tells us that $16+q\ge4\times7-10$, i.e. $F$ has two quadrilateral faces. They cannot be disjoint, as $F$ has only seven vertices. They cannot intersect at a vertex; otherwise all the vertices of $F$ would be contained in a 3-dimensional affine subspace. So they must intersect at an edge, and then their union contain six vertices, and is contained in a 3-dimensional affine subspace. Since there is only one vertex outside this subspace,  $F$ is again a pyramid over a 3-polytope $R$ with six vertices and ten edges (in fact, $R$ is $\TA$). In both cases, the other facet $F_1$ containing $R$ is also a pyramid, and thus,
and $F\cup F_1$ contains 22 edges. Hence there are 13 edges outside  $F\cup F_1$ which are incident to five vertices of degree at least five, which is again impossible.

So the facet $F$ has at least 17 edges. Denote  by $e_b$ the number of edges between the the six vertices outside $F$ and the facet $F$, and by $e_a$ the number of edges among the six vertices outside $F$. Then  $e_b\ge 7$. On the other hand, $e_a+e_b\le 18$ and $2e_a+e_b\ge 6\times 5$, implying $e_a\ge 12$ and $e_b\le 6$, a contradiction.

{\bf Case 6.}  Every facet, including $F$, has six or five vertices.

It is well known  (see e.g. \cref{thm:triplexes}) that a 4-polytope $F$ with six vertices has at least 13 edges, and can have exactly 13 edges only if $F=M_{2,2}$, i.e. a two-fold pyramid over a quadrilateral. We will show first that this  case does not arise. Suppose $F$ is such a facet, and consider the other facet $F_1$ containing the pyramid $R$ over a quadrilateral. Then   $F_1=M_{2,2}$, and $F\cup F_1$ contains 18 edges.  Hence there are 17 edges outside  $F\cup F_1$ which are incident to six vertices of degree at least five. Let $S$ denote the set of the six vertices outside $F\cup F_1$. Denote  by $e_b$ the number of edges between $S$ and $F\cup F_1$, and by $e_a$ the number of edges among the six vertices in $S$. Then $e_a+e_b= 17$ and $2e_a+e_b\ge 6\times 5$, implying $e_a\ge 13$ and $e_b\le 4$. However, having $e_b\le 4$ contradicts Balinski's Theorem, since removing at most four vertices  in $S$, those incident to the edges counted in $e_b$, disconnects the polytope graph.

Thus any facet $F$ with six vertices has at least 14 edges.  Since ${6\choose2}=15$, there is at most one pair of vertices in $F$ which is not joined by an edge. Thus no 2-face of $F$ can be a quadrilateral, pentagon etc, so $F$ is 2-simplicial, i.e. every 2-face  is a simplex.

It is clear that any facet with five vertices is  a simplex, and hence also  2-simplicial.

This all implies that $P$ is 2-simplicial, and therefore the lower bound theorem for 2-simplicial polytopes \cite[Thm.~1.4]{Kal87} gives the final contradiction as it ensures that $P$ would have at least 50 edges.
\end{proof}

\end{document}